\documentclass[12pt]{amsart}

\usepackage{amssymb,amscd,amsthm,extarrows,verbatim,amsmath,color,fancyhdr,mathrsfs}
\usepackage{amsaddr}
\usepackage{graphicx}
\usepackage{turnstile}
\usepackage{arydshln}
\usepackage{enumitem}
\usepackage{tikz}
\usetikzlibrary{arrows}

\usepackage[letterpaper, left=2.5cm, right=2.5cm, top=2.5cm,
bottom=2.5cm,dvips]{geometry}

\newcommand{\E}{\mathbf{E}}
\newcommand{\R}{\mathbf{R}}
\def\arrvline{\hfil\kern\arraycolsep\vline\kern-\arraycolsep\hfilneg}

\newtheorem{thm}{Theorem}[section]
\newtheorem{prop}[thm]{Proposition}
\newtheorem{lem}[thm]{Lemma}

\newtheorem{cor}[thm]{Corollary}
\newtheorem{ex}[thm]{Example}
\newtheorem{defn}[thm]{Definition}

\title{K-Nearest Neighbor Approximation
Via the Friend-of-a-Friend Principle}
\author{Jacob D.\ Baron \hspace{.8cm} R.W.R.\ Darling}
\address{National Security Agency, Fort George G.\ Meade, MD 20755-6844, USA}
\date{\today}

\begin{document}

\maketitle

%\vskip.5in

\begin{abstract}
Suppose $V$ is an $n$-element set where for each
$x \in V$, the elements of $V \setminus \{x\}$ are ranked by
their similarity to $x$.
The $K$-nearest neighbor graph 
is a directed
graph including an arc from each $x$ to the $K$
points of $V \setminus \{x\}$ most similar to $x$. Constructive
approximation to this graph using far fewer than $n^2$ comparisons
is important for the analysis of large high-dimensional data sets. 
\emph{$K$-Nearest Neighbor Descent} is a parameter-free heuristic where a sequence of
graph approximations is constructed, in which second
neighbors in one approximation are proposed as neighbors in the next. Run times in a test case fit an $O(n K^2 \log{n})$ pattern. This bound is rigorously justified for a similar algorithm,
using range queries, when applied to a homogeneous 
Poisson process in suitable dimension. However 
the basic algorithm fails to achieve
subquadratic complexity on sets whose similarity rankings
arise from a ``generic'' linear order on the $\binom{n}{2}$ inter-point distances in a metric space.
\end{abstract}

{\small
\textbf{MSC class: } Primary: 90C35; Secondary: 06A07
}

\vspace{.4in}
\begin{center}
    \emph{I get by with a little help from my friends.}
\end{center}
\vspace{.15in}
\hspace{2.5in} --- John Lennon and Paul McCartney, 1967
\vspace{.2in}

%\newpage
%\tableofcontents
%\newpage

\section{$K$-Nearest Neighbor Approximation}
\subsection{Motivation}
Approximation to the $K$-nearest neighbor graph
for a data set of size $n$ is needed before machine learning algorithms can be applied. Typically $K$ is small (say 5 to 50) but $n$ may be in millions or billions, rendering exhaustive search impractical.
Muja and Lowe \cite{muj} write: ``the
most computationally expensive part of many computer vision and machine learning algorithms consists of finding nearest neighbor
matches to high dimensional vectors that represent the training data.''
Examples include the popular DBSCAN clustering algorithm \cite{est},
nearest neighbor classifiers described in
Devroye, Gy{\"o}rfi \& Lugosi \cite{dev}, and UMAP \cite{mci}.

\subsection{Metric space versus ranking approaches}
\label{subsec1.2}
Metric spaces provide a natural context for similarity rankings. In a finite subset $X$ of a metric space $(\mathcal{X}, \rho)$ where for each $x \in X$ the $|X|-1$ distances from $x$ to the other points are distinct\footnote{
This is a weaker than requiring all $\binom{|X|}{2}$ inter-point
distances to be distinct.},
\begin{comment}
\begin{align}
\label{eq:uniqueDistsByPt}
\delta(x,y) \neq \delta(x,z) \; \text{ unless } \; y=z
\end{align}
(for each $x,y,z \in \mathcal{X}$),
\end{comment}
%with all $\binom{|\mathcal{X}|}{2}$ distances distinct, 
$\rho$ induces a similarity ranking for each point in an obvious way: $x$ is more similar to $y$ than to $z$ iff $\rho(x,y)<\rho(x,z)$. However the $K$-nearest neighbor concept does not require a metric space.
It suffices to have an oracle which determines,
for distinct points $x, y, z$, whether $y$ is more similar to $x$ than $z$
is\footnote{For ordinal data analysis in general, see
Kleindessner \& von Luxburg \cite{kle}.
}. In other words, for each $x$, the oracle knows a ranking of the other points by their similarity to $x$. When two different metrics may be applied to the same set of points,
and their induced similarity rankings coincide, any method that proposes two
different approximate $K$-nearest neighbor graphs 
violates common sense, rather like a computation in
 differential geometry which gives different results in
 different co-ordinate systems.

Given a finite subset $S$ of $\R^d$ with the $L_p$ metric, for 
$1 \leq p \leq \infty$, the balanced box-decomposition tree
of Arya et al.\ \cite{ary}, a variation of the classical \emph{k-d tree} reviewed in
\cite{dev},
constructs an approximate 
$K$-nearest neighbor graph in $O(d K n \log{n})$ steps,
with a data structure of $O(d n)$ size.
Muja and Lowe \cite{muj} propose other tree-based algorithms.
By using random projections, Indyk \& Motwani's \cite{ind}
\emph{locality sensitive hashing} has
$O(K n^{1 + \rho} \log{n})$ query time and
requires $O(d n + n^{1 + \rho})$ space, for a problem-dependent
constant $\rho > 0$ described in Datar et al.\ \cite{dat}, and
requires tuning to the specific problem.

Houle and Nett \cite{hou} describe a
rank-based similarity search algorithm called the \textit{rank cover tree},
which they claim outperforms in practice both the
balanced box-decomposition tree and locality sensitive hashing.
Haghiri et al \cite{hag} propose another
comparison-tree-based nearest neighbor search.
Tschopp et al \cite{tsc} present a
randomized rank-based algorithm using the 
``combinatorial disorder'' parameter of Goyal et al \cite{goy}.

\subsection{$K$-nearest neighbor descent}

In \cite{don},
Dong, Charikar, and Li proposed and implemented 
a heuristic 
called \textbf{$K$-nearest neighbor descent} (NND)
for approximation to the $K$-nearest neighbor graph.
In the metaphor of social networks, their principle is:
\begin{align*}
    \textit{A friend of a friend could likely become a friend.}\footnotemark 
    \tag{FOF}
\end{align*}
\footnotetext{Here ``likely'' should be read as ``more likely than the average person,'' not that the probability is necessarily greater than 50\%.}
We call this the \textbf{\emph{friend-of-a-friend principle} (FOF)}. NND has an appealing
simplicity and generality:
\begin{itemize}[leftmargin=.5in]
    \item 
    No parameter choices, except $K$.
    \item
    No precomputed data structures, such as trees or hash tables.
    \item
    Simple, concise code; only the similarity oracle is 
    application-specific.
        \item 
    No vector embedding is required.
    \item
    Randomly initialized -- repeated NND runs give
    alternative approximations.
\end{itemize}
Perhaps for these reasons, NND was the method of choice
in the general-purpose UMAP dimensionality reduction
algorithm \cite{bec, mci} of McInnes, Healy, et al.\

The main drawback of NND has been its lack of any theoretical justification.
The present work seeks to shed light on this issue.
Various software efficiency  optimizations render the algorithm as 
presented in \cite{don} opaque to rigorous analysis. 
Instead we analyze non-optimized versions and variations which
preserve the essential concept, which is FOF. In \cite{darnnd},
we report on observed complexity of a Java Parallel Streams
implementation.

\subsection{Outline of the paper}
We begin by describing NND and several related $K$-nearest neighbor approximation algorithms that take advantage of FOF (Section \ref{s:algs}). Next we give examples that show how these algorithms behave in different contexts, succeeding in some and failing in others (Section \ref{s:examples}).

We develop the ranking framework in detail, and 
use it to demonstrate the failure of FOF-based algorithms
to achieve sub-quadratic complexity in metric spaces which arise from generic \textit{concordant ranking systems}
(Section \ref{s:ranking}).

Finally we use random graph methods to prove that 
one FOF-based algorithm succeeds for a homogeneous Poisson 
process on a compact metric space with only $O(n \log{n})$ work---better than the $O(n^{1.14})$ claimed in \cite{don} (Section \ref{s:knnpp}).

%A main theme is the striking dichotomy between the similarity measures for which FOF succeeds marvelously and those for which it fails completely. %We explore this dichotomy throughout.

\begin{comment}
The number of distance evaluations of the algorithm on the torus
actually \textit{decreases} with dimension:
Corollary \ref{c:2nfwork} below indicates that the number
of operations is of order
\[
\frac{n K^2 \log{(n / K )}} {\sqrt{d \log{K}}}
\]
in dimension $d$, provided $n / K $ exceeds $\exp{(\sqrt{d \log{K}})}$,
not $O(n^{1.14})$ as \cite{don} proposes. However each evaluation
involves a distance computation which is $O(d)$ complexity,
meaning that overall complexity increases as $O(\sqrt{d})$.
\end{comment}

%%%%%%%%%%%%%%%%%%%%%%%%%%%%%%%%%%%%%%%%%%%%%%%%%%%

\section{$K$-Nearest Neighbor Approximation Algorithms Exploiting the Friend-of-a-Friend Principle}
\label{s:algs}

\subsection{Friends and cofriends}
Let us first formalize two notions we have used already. For a positive integer $m$, $[m]$ refers to the set $\{1, 2, , \ldots, m\}$.

\begin{defn} \label{d:ranksystem}
A \emph{ranking system} $\mathcal{S}=(S,(r_x)_{x\in S})$ is a finite set $S$ together with, for each $x \in S$, a ranking $r_x: S \setminus \{x\} \rightarrow [|S|-1]$ of the other points of $S$. We say that $x$ \emph{prefers} $y$ to $z$
%, or $x$ is more similar to $y$ than to $z$, 
iff $r_x(y) < r_x(z)$ (for distinct $x,y,z \in S$). We typically abbreviate $(r_x)_{x\in S}$ by $\mathbf{r}$.
\end{defn}

\begin{defn} \label{d:knngraph}
For a ranking system $\mathcal{S}=(S,\mathbf{r})$ and a positive integer $K < |S|$, the \emph{$K$-nearest neighbor graph} for $\mathcal{S}$ is the directed graph on $S$ that contains, for each distinct $x,y\in S$, an arc from $x$ to $y$ iff $r_x(y) \leq K$. 
\end{defn}

%\subsection{Framework}
All the $K$-nearest neighbor approximation algorithms we will describe follow a common framework. They take as input a set $S$ and an integer $K$. In the background is a vector $\mathbf{r}$ of rankings that makes $S$ into a ranking system $\mathcal{S}=(S,\mathbf{r})$. We do not know $\mathbf{r}$ (if we did, we would be done), but we are allowed to query it, in the following sense. For any distinct $x,y,z \in S$, we may ask whether $r_x(y) < r_x(z)$ or vice-versa\footnote{See \cite{darnnd} for an extension of these ideas to rankings with ties.}. We are never told the value of any $r_x(y)$.  

The intermediate state of the algorithm is a directed graph on $S$ which represents our current best approximation to the true $K$-nearest neighbor graph for $\mathcal{S}$. Our initial approximation is a uniformly random $K$-out-regular digraph on $S$; that is, each point independently chooses 
a uniform random set of 
$K$ initial out-neighbors.
In successive iterations called \emph{rounds}, we allow points to ``share information'' in some way that, we hope, leverages FOF.

%% START HERE 9.18.20

At each round $t$, each point $x$ will be the source of some arcs (usually $K$ of them, except in Section \ref{s:2nrq-describe}); we write $F_t(x) \subseteq S \setminus \{x\}$ for the set of targets of these arcs, which we call the \textbf{\emph{friends}} of $x$. The points that view $x$ as a friend are called the \textbf{\emph{cofriends}} of $x$, denoted 
\[C_t(x):= \{ u \in S \setminus \{x\}: x \in F_t(u) \}
\]
(possibly $C_t(x) = \varnothing$). To leverage FOF, we perform some operation intended to ``make introductions'' between certain pairs of points which FOF suggests might ``get along'' (i.e.\ rank each other highly). As a result of these ``introductions,'' one (or possibly all, or many) point(s) $x$ updates her friend set if she has just ``met'' some new point $y$ whom she prefers to some current friend $z$. After these updates, we move to round $t+1$ and repeat. When some condition is met, we stop and return our approximate $K$-nearest neighbor graph.

In a \textbf{friend list request}, a point $x$ asks some other point $y$ to provide $x$ with $F_t(y)$, $y$'s current list of friends. Typically $y \in F_t(x)$, which leads $x$ to think $y$ might have friends that $x$ would like to meet. Then, via queries to $r_x$, $x$ determines her $K$ most-preferred elements of $F_t(x) \cup F_t(y) \setminus \{x\}$. Finally, $x$ updates her friend set to consist of these $K$ points.

A \textbf{friend barter} is a pair of reciprocal friend list requests, shown in Figure \ref{f:friendpool}. A point $x$ makes a friend list request of another point $y$, and vice-versa, and each updates her respective friend set as appropriate. Typically $y \in F_t(x)$, but this does not mean $x \in F_t(y)$. Thus $x$ is introduced to friends of a friend, while $y$ is introduced to friends of a cofriend.
%shown where the solid lines meet, asks one of its friends $v$ for a list of $v$'s friends, while providing a list of its own friends to $v$. Both $v$ and $v'$ update their friend set as follows; $v'$ picks the $K$ nearest elements of the pooled list, and likewise for $v$.

\begin{figure}
\caption{ \textbf{Friend barter: }\textit{
Out of twenty points in the left pane, we show two, each with six friends, indicated by dashed and solid arrows, respectively. Point $y$ is a friend of point $x$. The right pane shows the new friend sets of $x$ and $y$, respectively, after the friend barter. Point $y$ gained three new friends (and discarded three), while point $x$ upper gained two (and discarded two).}
} \label{f:friendpool}
\begin{center}
\scalebox{0.42}{\includegraphics{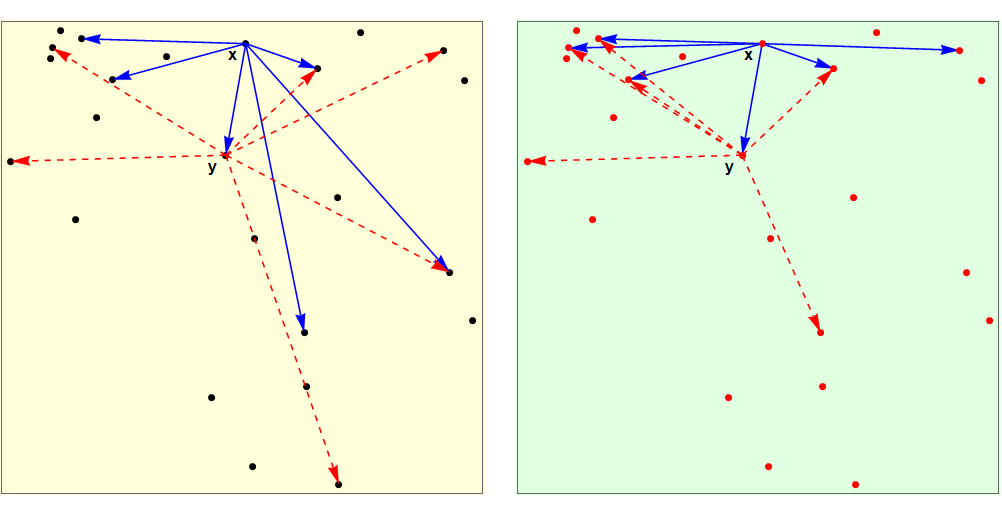}} 
\end{center}
\end{figure}

Within this framework of iterated ``introductions'' and updates, specifying an algorithm amounts to deciding which introductions occur in a given round and fixing a stopping condition. Shortly we will specify several algorithms in this way, but before doing so, it is worth reflecting on the initialization step. There is good reason to initialize randomly, quite apart from convenience.

\subsection{Consequences of random initialization} \label{s:randominit}
Let $D$ be any $K$-out-regular directed graph on $S$, and let $G$ be the corresponding undirected graph, with mean vertex degree about $2 K$, obtained by ignoring the direction of arcs in $D$. Consider initializing a FOF-based algorithm with $D$. In any FOF-based algorithm, information travels, via introductions to friends of friends and/or friends of cofriends, by at most 1 step in $G$ per round. Thus if the shortest path in $G$ from $y$ to $x$ has length $d$, any FOF-based algorithm will need to run at least $d-1$ rounds before knowledge of $y$ reaches $x$.

The number of rounds needed for every point to have the possibility
of learning about its true nearest neighbors is thus related to the
\emph{diameter} of $G$, i.e.\ the maximum over
$x, y \in S$ of the length of the shortest path from $x$ to $y$. We would like this diameter to be small.

On the other hand, since we initially have no knowledge of $\mathbf{r}$,
%(the vector of rankings $(r_x)_{x\in S}$), 
we should not bias our algorithm by arranging to include, say, an $(n-1)$-leaf star in $G$. Each round involves sorting operations
at each point, and to keep these sorts to $O(K^2 \log K)$ time,
the degrees in $G$ need to be kept at $O(K)$.

%Moreover, we would like to keep degrees in $G$ to $O(K)$. This is because in some FOF-based algorithms, each round's update step involves sorting a list at each point, where the list in round $t$ at point $x$ has size about $\max\{K\,|F_t(x) \cup C_t(x)|,n\}$. We would like to keep these sorts to $O(K^2 \log K)$ time each, which occurs iff $|C_t(x)| = O(K)$ for all $x$. But if $x$ is the center of (say) an $O(n)$-leaf star in $G$, then the sort at $x$ in round 1 takes $O(n \log n)$ time alone.

Taking $D$ uniformly random strikes a good balance. Both the diameter of, and degrees in $G$ are then small. Each point has only a Binomial$(n-1, \frac{K}{n-1})$ number of cofriends. 
In Lemma \ref{p:diameter} of the Appendix, we will
prove that, for $K \geq 3$, with probability tending to 1 as $n \to \infty$, $G$ is an \emph{expander graph} in the sense of \cite{kri},
implying that its diameter is $O(\log n)$. 

\section{Implementations of $K$-nearest Neighbor Descent} \label{s:implement}

\subsection{Functional implementation for ranking systems} \label{s:javannd}
Darling has published an efficient Java functional implementation of NND, for application to the \texttt{prank2xy} algorithm \cite{darp}. 
The ranking system notion $r_x(y) < r_x(z)$ is
expressed abstractly in terms of a Java \texttt{Comparator}.
Batch updates to the $K$-NN approximation are made
in parallel using a \texttt{Collectors} method of a Java \texttt{ParallelStream}.

In experiments (see \cite[Appendix]{darnnd}) using  Kullback-Leibler divergence
between points on a
$d$-dimensional simplex in $\R^{d}$,
our statistical convergence criterion for NND was always satisfied
within $\lceil 2 \log_K{n} \rceil$ batch updates.
In these tests, 
$2 \times 10^4 \leq n \leq 2 \times 10^6$, $K \in \{16, 32, 64\}$, and $d \in \{10, 20, 40, 60\}$. Satisfactory approximation was achieved for $K > d$.

We conjecture that $\lceil 2 \log_K{n} \rceil$ is the correct bound
for the number of NND rounds in the ill-defined set of 
cases where the algorithm works.

\subsection{Original implementation for metrics} \label{s:batchwise}
A sophisticated algorithm for $K$-nearest neighbor descent
on points with a symmetric distance function was designed
and implemented in \cite{don}. We shall not describe it in detail, except
to say that friend sets are not immutable during a round, and the
stopping criterion is different to ours. The authors
describe in detail the application to five well-studied data sets 
of sizes between 28,755 and 857,820 using
various metrics for similarity. 

\subsection{Python implementation}
McInnes has published the \texttt{pynndescent} Python library.
It assumes that similarity is obtained from a metric, of which 22
examples are available in the code. This is one of 19 single-threaded
approximate $K$-NN Python algorithms among benchmarks at \cite{ber}.

\subsection{Scheduled pointwise NND} \label{sec:pointwise}
Apply an arbitrary ordering to the $n$ elements of $X$, called the schedule.
Initial friend sets $\{F_0(x)\}_{x \in X}$ are uniformly random as above.
Pass $t$ consists of a set of $n$ friend barters: visit each element of $x \in X$ in the schedule, 
and update the friends of $x$ by friend barters with all friends of $x$.
That means $x$ replaces its former friend set $F_{t-1}(x)$
by a new friend set $F_{t}(x)$, consisting of the $K$
 top ranked elements of the set
 \[
 F_{t-1}(x) \cup \left(
 \bigcup_{y \in F_{t-1}(x)} F_{t-1}(y)
 \right) \setminus \{x\}.
 \]
Perform multiple passes until no further changes occur. Note
that one pass is not the same as one round of the functional
implementation of Section \ref{s:javannd}, in which all friend sets
are updated simultaneously by a \texttt{collect(.)} method. This
Scheduled pointwise NND was not implemented by us, but is
used elsewhere is this paper for theoretical analysis.

%%%%%%%%%%%%%%%%%%%%%%%%%%%%%%%%%%%%%%%%%%%%%%%%%%%%%%%%%%%%%%%%%%
\section{Examples}\label{s:examples}

This section considers how FOF-based algorithms fare against 
heterogeneous examples of ranking systems, summarized in Table \ref{t:contexts}.
We mostly restrict attention 
%for pedagogical purposes 
to ranking systems induced by distance in a metric space $(\mathcal{X},\rho)$, as in Section \ref{subsec1.2}. In each such example, the point set of the ranking system will be called $X \subseteq \mathcal{X}$, and $|X|=: n$. For any triples $x,y,z \in X$ with $y \neq z$ but $\rho(x,y)=\rho(x,z)$, we can choose arbitrarily whether $r_x(y) < r_x(z)$ or vice-versa.

\begin{table}
\caption{\textit{
Some Contexts for $K$-NND}
}
\centering
 \begin{tabular}{| l| c | c |} \label{t:contexts}
Point Generation & Metric & FOF Applies?  \\ [0.5ex]
 \hline
Leaves of Star Graph & Paris Metric & Yes \\ \hline
Random Points on Circle & Path Metric & Yes \\ \hline
Powers of 2 & $\ell_1(\R)$ & Unknown \\ \hline
Random $m$-Strings & Longest Common Subsequence & No ($m$ large) \\ \hline
\end{tabular}
\end{table}

\subsection{Paris metric} \label{s:starleaves}

Let $\mathcal{X}$ be $\ell_1^n$, meaning the $n$-dimensional
real Banach space where
\[
\|(a_1, \ldots, a_n)\| = \sum_1^n |a_i|.
\]
Let $e_i$ denote the $i$-th basis vector. Take
$X$ to consist of $n$ points of the form
\[
\{x_i:=\eta_i e_i\}_{1 \leq i \leq n},
\]
where $0 < \eta_1 < \cdots < \eta_n$ are real numbers.
Then
\[
\rho(x_i,x_j)=\|x_i - x_j\|= \eta_i + \eta_j, \quad i \neq j.
\]
Project $X$ and an auxiliary point $x_o$ into the plane, such that 
$x_o$ maps to the origin, no three points are co-linear,
and $x_i \in X$ is distance $\eta_i$ from $x_o$, as in Figure 
\ref{f:stargraph}; then $\rho$ is called the \emph{Paris metric}.
Using the metropolis as a metaphor, the shortest path
between $x_i$ and $x_j$ is the sum of their distances
from the origin. In the graph theory context, view $x_i$ as a leaf node
of an edge-weighted star centered at $x_0 \notin X$,
where edge $x_0 x_i$ carries weight $\eta_i$.
The Paris metric is the path metric\footnote{
Given a pair of vertices of an edge-weighted graph, the
 path metric is the minimum, over paths between those vertices,
 of the sum of the weights along a path.
} on this star.

For each $j > K$, the $K$-nearest neighbor set of $x_j$
is the set
\[
\{x_1, x_2, \ldots, x_K\}.
\]
For $j \leq K$, it is
\[
\{x_1, x_2, \ldots, x_{K+1}\} \setminus \{x_j\}.
\]

\begin{figure}
\caption{\textbf{Paris metric: }
Let $X$ be the leaf nodes of the star graph shown
embedded in the plane. The distance between two leaves
is the sum of their distances from the central
node. Here $K = 4$, and 8 of the vertices (dashed arcs)
have the same set of $K$ nearest neighbors, which
are the $K$ nodes (solid arcs) closest to the center.
} \label{f:stargraph}
\begin{center}
\scalebox{0.4}{\includegraphics{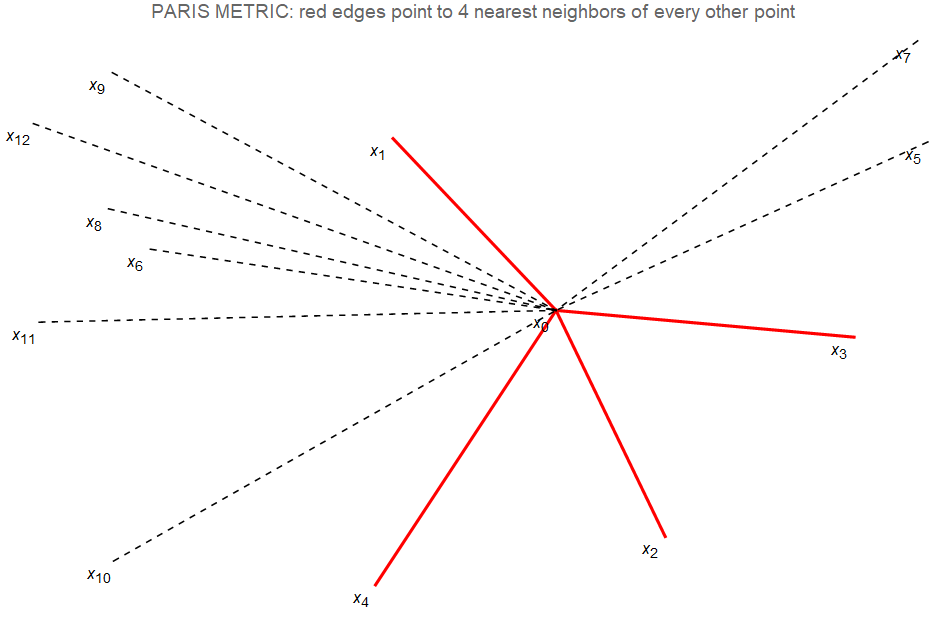}}
\end{center}
\end{figure}

%Batchwise %$K$-nearest neighbor descent
NND works well in this example because every pair of vertices are in perfect agreement about how they rank the other $n-2$ vertices. Thus $x_i$'s information is always useful---or at least trustworthy---to $x_j$, and vice-versa. Each tells the other about the vertices she knows that are closest to $x_0$. 
%Fix $\delta \in (0,1)$, and 

Suppose our goal is to discover the $K$-nearest neighbor graph exactly,
within some specified number of rounds. 
Proposition \ref{p:diameter}, which bounds the diameter of the initial graph, 
implies that if $K \geq 3$, 
\[
r^*:=\left\lceil \log_{K-1} (n) \right\rceil
\]
rounds of batchwise NND suffice for large $n$, with high probability\footnote{
Given $\epsilon > 0$, there exists $n_0$ such that the event occurs with probability greater than $1 - \epsilon$ when $n > n_o$.}.

%%%%%%%%%%%%%%%%%%%%%%%%%%%%%%%%%%%%%%%%%%
Suppose $i \leq K$, and consider all paths
in the initial undirected graph
of length $r^*$, starting at $x_i$ . By Proposition \ref{p:diameter}, every point is on one of these paths with high probability. %$(V,E)$. 
As the algorithm progresses, every other point along each of these paths will retain $x_i$ in its friend set. Hence every other point will have
$x_i$ in its friend set after $r^*$ rounds.

\subsection{Random points on a circle}
Suppose $\mathcal{X}$ is the circle $S^1$, and the distance between two points is the angle subtended in radians (i.e. path distance).
Let $X$ consist of $n$ points sampled
independently according to Lebesgue measure on
the circle, or possibly a Poisson$(n)$ number of points. 
Spatial intuition suggests there 
may be a constant $\gamma \in (0, 1)$, depending on $K$,
such that each of the first $T=O(\log{(n/K)})$ rounds of friend set updates
multiplies the expected
angular span of the friend set of $x$ by $\gamma$ or less.
If this is true, then after $T$ rounds of
friend set updates, elements of
the friend set $F_T(x)$ typically lie within a span
of $2\pi\gamma^{T}$ radians. The $K$ nearest neighbors
of a point typically lie within a span of $2 \pi K / n$
radians. If NND works at all in this context, we would
expect
\[
T:=\lceil (\log{(n/K)})/\log{(1/\gamma)} \rceil
\]
rounds to expose these nearest neighbors, after which further friend updates
would have no effect. We have been unable
to make this argument rigorous for NND, but
Section \ref{s:knnpp} provides a  rigorous analysis
along these lines for a range-based, rather than a ranking-based,
algorithm. The value of $\gamma$ derived in (\ref{e:scaleparam}) for that algorithm is
\(
  \gamma:=1 - \sqrt{1 - \frac{2} {K}},
\)
provided $K \geq 3$.

\subsection{Random ranking system}
\label{s:uniformrandomRS}
In this example there is no metric. Given a set $X$ of size $n$, simply choose a ranking system on $X$ uniformly at random. Do this by choosing for each $x \in X$ a uniformly random ranking $r_x$ of $X \setminus \{x\}$, with all choices independent. 

There is no hope for decent performance from any FOF-based algorithm in this context, because the FOF principle no longer applies. 
Since $r_x, \, r_y$ are independent for distinct $x,y$, a point that $y$ ranks highly is no more or less likely to be ranked highly by $x$ than is any other point, regardless of how $x$ and $y$ rank each other.

Consider applying the scheduled pointwise NND of 
Section \ref{sec:pointwise}, using only friend list requests. (The analysis would be similar for any other NND variant, if a little less clean.) Consider the progress of the algorithm from the point of view of some fixed $x$. How long until $x$'s friend set contains, say, $\lceil K/2 \rceil$ of her true best friends $B:=\{y \in X \setminus \{x\} \;:\; r_x(y) \leq K\}$? There is an update at $x$ only every $n$th round, and in these rounds $x$ meets at most $K^2$ points, some not for the first time. Each new point $z$ is equally likely to be ranked in any position by $r_x$. Thus we expect $x$ to have to meet about $n/2$ points, over the course of at least $n/(2K^2)$ relevant rounds, before she finds $\lceil K/2 \rceil$ points of $B$. Running through $n/(2K^2)$ relevant rounds means running $n^2/(2K^2)$ rounds total, taking $O(n^2\log K)$ work. Essentially, $x$ is performing an inefficient version of exhaustive search during her rounds. Since $x$ was arbitrary, all other points have the same experience. 

It is fair to characterize this performance as worse than exhaustive search.

\subsection{Longest common substring}
\label{s:lcss}
Take a finite or countable alphabet $\mathcal{A}$, and a large positive integer $m$.
Consider the metric space $(\mathcal{X},\rho)$ consisting of strings
$\mathbf{x}:=x_1 x_2 \cdots x_{m}$ of length $m$, with characters $x_i \in \mathcal{A}$,
under the metric based on the \emph{longest common substring}.
To be precise, the length $M(\mathbf{x, x'})$ of the longest common substring
between two strings is given by
\[
M(\mathbf{x, x'}):=\max \{s: \exists \,i, j \mbox{ such that }
x_{i+1} x_{i+2} \cdots x_{i+s} = x'_{j+1} x'_{j+2} \cdots x'_{j+s}\},
\]
and the metric itself is defined by
\[
\rho(\mathbf{x, x'}):= 1 - \frac{M(\mathbf{x, x'})}{m}.
\]
Symmetry $\rho(\mathbf{x, x'}) = \rho(\mathbf{x', x})$ is immediate,
as is 
$\rho(\mathbf{x, x}) = 0$, while the triangle inequality is verified 
in a note of Bakkelund \cite{bak}.

Let $\mu$ be a non-trivial probability measure on $\mathcal{A}$.
Let $X \subset \mathcal{X}$ 
consist of $n$ strings of the form
$\mathbf{x}:=X_1 X_2 \cdots X_{m}$, where each $X_i$ is sampled independently from
$\mathcal{A}$ according to the distribution $\mu$, with distinct $\mathbf{x,y}\in X$ independent. 

Given three random strings, ties such as
$M(\mathbf{x,y}) = M(\mathbf{x,z}) = q$ will be fairly common.
We break ties by ranking
\( r_\mathbf{x}(\mathbf{y}) < r_\mathbf{x}(\mathbf{z})
\)
if the product of probabilities of the characters in the
longest substring that $\mathbf{y}$ shares with $\mathbf{x}$
is more than the corresponding product for $\mathbf{z}$ and $\mathbf{x}$. If there are still ties, break these arbitrarily.

%Suppose that $\mathbf{X}$ and $\mathbf{Y}$ are random strings sampled  independently in this way.
Given an integer $K$ much smaller than $n$,
there is an integer $q_K$ (also depending on $m$ and $\mu$) such
\[
\Pr[M(\mathbf{x,y}) \geq q_K + 1] < \frac{K}{n} 
\leq \Pr[M(\mathbf{x,y}) \geq q_K].
\]
The importance of $q_K$ is that the $K$ nearest neighbors of $\mathbf{x}$
in $X$ lie within the intersection with $X$ of the ball
of radius $1 - q_K/m$, centered at $x$, which may be written
\[
R(\mathbf{x}):=\{\mathbf{y} \in X: M(\mathbf{x,y}) \geq q_K\}.
\]

Arratia, Gordon \& Waterman \cite[Theorem 2]{arr} provide an asymptotic estimate
of the value of $q_K$ for large $m$. We will not state their theorem
precisely, and we omit some intermediate calculations, but the key estimate depends on the sum $p$ of squares of probabilities of characters in the alphabet:
\begin{align*} \label{e:ssquprobs}
 p:= \sum_{a \in \mathcal{A}} \mu(\{a\})^2 \in (0,1).
\end{align*}
\begin{comment}
. Taking logs to base $1/p$, for $p$ as in (\ref{e:ssquprobs}),
define a constant
\[
\theta:= \log_{1/p} (m^2 (1-p).
\]
For any value of $c$ for which $q_K:= c + \theta$ is an integer,
\cite[Theorem 2]{arr} gives a precise sense to the statement:
\[
\Pr[M(\mathbf{X,Y}) \geq q_K] \approx 1 - \exp{(-p^c)}.
\]
meaning that the difference between right and left sides tends to zero
uniformly as $m \to \infty$. We want the right side to be $K/n$,
which means that we take:
\[
c \approx \log_{1/p} (n/K)
\]
after using the approximation $-\log (1 - K/n) \approx K/n$.
Adding the two logarithms, 
\end{comment}
The $K$ nearest neighbors
of a specific string $\mathbf{x}$ will typically
be strings whose longest common substring
with $\mathbf{x}$ has length at least
\[
q_K = \log_{1/p} (m^2 (1-p)) + \log_{1/p} (n/K)
= \frac{2 \log(m) + \log ((1 - p) n/K)}{- \log{p}}.
\]
On the other hand, replacing $K$ by 1, the
nearest neighbor of $\mathbf{x}$ will typically
share a common substring with $\mathbf{x}$ 
of length $q_1$, where
\[
q_1 = 
\frac{2 \log(m) + \log ((1 - p) n)}{- \log{p}}.
\]

Notice that %(for reasonable $K,n,m$) we have 
$q_K \approx q_1 \ll m/K$. Thus for any distinct $\mathbf{x,y,z}\in X$, even if $\mathbf{y}$ and $\mathbf{z}$ are ranked highly by $\mathbf{x}$, the longest substring that $\mathbf{x}$ shares with $\mathbf{y}$ will typically be disjoint from that shared between $\mathbf{x}$ and $\mathbf{z}$. In other words, knowing that $\max(r_{\mathbf{x}}(\mathbf{y}), r_{\mathbf{x}}(\mathbf{z})) \leq K$ makes it virtually no more likely that $r_{\mathbf{y}}(\mathbf{z}) \leq K$, so $\mathbf{x}$'s information is virtually irrelevant to $\mathbf{y}$ even if $\mathbf{y}$ is a good friend of $\mathbf{x}$, and \textit{vice versa}, replacing friend by cofriend\footnote{The situation is not symmetric: $\mathbf{y}$ is a friend of $\mathbf{x}$ while $\mathbf{x}$ is a cofriend of $\mathbf{y}$.}.

This is precisely the problem we encountered for a random ranking system, if a bit less extreme. The friend of a friend principle is inapplicable 
for parameters such as the next example.

\textbf{Numerical example: }
In ``big data'' applications, we may have 
$m$ in the thousands, while $q_K$ may be in double digits.
For example, if $m = 2^{16}$, $n = 2^{33}$, $K = 2^5$, and $p = 2^{-4}$,
then $q_K = 15$ and $q_1 = 16$. In other words,
the $32$ nearest neighbors of some string $\mathbf{x}$
will typically share common substrings with
$\mathbf{x}$ of length $15$ to $17$.
Moreover these shared substrings will be different
for different neighbors, covering at most
a proportion $K q_K / m \approx 2^{-7}$
of the characters in $\mathbf{x}$.

\subsection{Powers of 2}
\label{sec:powersof2}
Take $\mathcal{X}=\textbf{R}$ with the usual distance, and let $X$ consist of the first $n$ nonnegative powers of 2. For example, the two nearest neighbors of 32 are 16 and 8. 
We do not know whether NND works here. It seems to lie somewhere between the opposite extremes occupied by the Paris metric of Section \ref{s:starleaves} and the longest common substring metric
of Section \ref{s:lcss}. For distinct $x,y$ with $r_x(y) \leq K$, there are on average about $K/2$ points $z$ for which $\max(r_x(z), r_y(z)) \leq K$. The corresponding quantity for the Paris metric is $K-1 \approx K$; for longest common substring it is $K/n \approx 0$.
We revisit the powers of 2 in Example \ref{ex:RankingSystemUnique}.

%%%%%%%%%%%%%%%%%% JAKE'S SECTION ON RANKINGS

\section{Failure of NND}\label{s:ranking}

Our goal in this section is to explore what ``went wrong'' for NND in the longest common substring example. % of Section \ref{s:lcss}. 
Our intuition is that the rankings $r_\mathbf{x}$ and $r_\mathbf{y}$ for distinct $\mathbf{x},\mathbf{y}$ are sufficiently ``uncorrelated'' or ``independent'' that $\mathbf{y}$'s information is essentially useless to $\mathbf{x}$ even if $\mathbf{y}$ is a friend or cofriend of $\mathbf{x}$. The same problem arose in the case of a random ranking system in Section \ref{s:uniformrandomRS}. But while it is natural that this would happen for a random ranking system, it is much more surprising in the context of a %ranking system arising from a 
metric space, because intuitively, the triangle inequality should cause FOF to be helpful.

We now build a framework to formalize these ideas. With this framework in hand, we will prove the striking result that not only is a metric space insufficient to cause FOF to be helpful, but in fact in ``almost all'' metric spaces FOF fails to help.

\subsection{Metrizability and concordancy}
Our first step is to nail down just which ranking systems arise from a metric space.

\begin{defn}
A ranking system $\mathcal{S} = (S,\mathbf{r})$ is called \emph{metrizable} if there is a metric $\rho$ on $S$ that induces the rankings $\mathbf{r} := (r_x)_{x\in S}$, i.e.\ such that for each distinct $x,y,z \in S$, $\rho(x,y)<\rho(x,z)$ iff $r_x(y)<r_x(z)$.
\end{defn}

Goyal et al.\ \cite{goy} formalize ranking systems in a similar way,
but omit the concept of concordancy, which we now define. For a set $X$ and an integer $t$, let $\binom{X}{t}$ denote the set of $t$-element subsets of $X$. As in graph theory we frequently omit brackets for sets of size 2, e.g.\ $ab := \{a,b\}$.

\begin{defn}
Let $\mathcal{S} = (S,\mathbf{r})$ be a ranking system. For $x \in S$, let $\preccurlyeq_x$ be the linear order on
$S_x := \{ xy \;:\; y \in S \setminus \{x\}\}$ (here $xy$ is short for $\{x, y\}$)
defined by   
\[ \{x,r_x^{-1}(1)\} \preccurlyeq_x \{x,r_x^{-1}(2)\} \preccurlyeq_x \cdots \preccurlyeq_x \{x,r_x^{-1}(|S|-1)\}. \]
Formally\footnote{Recall that formally a partial order on $X$ is a (reflexive, antisymmetric, transitive) binary relation on $X$, i.e.\ a subset of $X \times X$. For example on $\{1,2,3,4\}$, the usual (linear) order $\leq$ \emph{is} the set $\{(1,1),(1,2),(1,3),(1,4),(2,2),(2,3),(2,4),(3,3),(3,4),(4,4)\}$. The statement ``$x \leq y$'' is nothing more than an abbreviation for ``$(x,y) \in \: \leq$.''}
$\preccurlyeq_x$ is a set of ordered pairs of %unordered pairs of 
elements of 
$S_x$; % \subseteq \binom{S}{2}$
for example $\preccurlyeq_x$ contains the element $(\{x,r^{-1}_x(2)\},\{x,r^{-1}_x(5)\})$ (because $2\leq 5$). %$\binom{S}{2}$ (the 2-element subsets of $S$)
% $\preccurlyeq_x = \{(\{x,r^{-1}(i)\},\{x,r^{-1}(j)\})\mid 1 \leq i < j \leq |S|-1\}$.)
Let $R_\mathcal{S}$ be the union of these sets: % binary relations:
\[ 
R_\mathcal{S} \;\; := \;\; \bigcup_{x \in S} \preccurlyeq_x \;; 
\]
this is a set of ordered pairs of elements of $\cup_{x\in S}S_x = \binom{S}{2}$. We call $\mathcal{S}$ \emph{concordant} if $R_\mathcal{S}$ is a subset of some partial order on $\binom{S}{2}$. In this case, the \emph{order type} of $\mathcal{S}$, denoted $\preccurlyeq_\mathcal{S}$, is the unique minimal partial order that contains $R_\mathcal{S}$ (i.e.\ the intersection of all partial orders containing $R_\mathcal{S}$). %, and we denote this partial order by .
\end{defn}

Henceforth we shorten ``concordant ranking system'' to \textbf{CRS}.

\begin{ex}[CRS]
\label{ex:ConcRS}
Figure \ref{f:hassediagrams} shows two views of the ranking system $\mathcal{S}$ on $S=\{a,b,c,d,e\}$ with rankings %given by
\begin{align*} (1,2,3,4) &= (r_a(b),r_a(e),r_a(d),r_a(c)) \\
&= (r_b(a),r_b(c),r_b(d),r_b(e)) \\
&= (r_c(d),r_c(e),r_c(b),r_c(a)) \\
&= (r_d(c),r_d(e),r_d(a),r_d(b)) \\
&= (r_e(c),r_e(a),r_e(d),r_e(b)).
\end{align*}
The left of Figure \ref{f:hassediagrams} shows, respectively, the linear orders $\preccurlyeq_a$ on $S_a$; $\preccurlyeq_b$ on $S_b$; $\preccurlyeq_c$ on $S_c$; $\preccurlyeq_d$ on $S_d$; $\preccurlyeq_e$ on $S_e$. (These are called \emph{Hasse diagrams}; an arrow appears iff the lower item is less than the upper item and there is no item in between.) Because there exists a partial order on $\binom{S}{2} = \cup_{x \in S}S_x$ that simultaneously extends all of these linear orders, $\mathcal{S}$ is concordant. The order type $\preccurlyeq_\mathcal{S}$ of $\mathcal{S}$ is the minimal such order, i.e.\ the one that contains no information beyond what is forced by $R_\mathcal{S}$. Its Hasse diagram is on the right of Figure \ref{f:hassediagrams}.
\end{ex}

\begin{figure}
\caption{ \textit{Concordant ranking system: Example \ref{ex:ConcRS}.
The five digraphs on the left show the rankings 
$r_a(\cdot), r_b(\cdot), r_c(\cdot), r_d(\cdot), r_e(\cdot)$,
respectively. The digraph on the right illustrates the unique
minimal partial order extending these rankings to
all $\binom{5}{2}$ point pairs.
 }
} \label{f:hassediagrams}
\begin{center}
\tikzstyle{arrow} = [thick,->]
\begin{tikzpicture}[node distance=1.5cm]

\node (phantom) {};

\node (a4) [right of=phantom] {$ac$};
\node (a3) [below of=a4] {$ad$};
\node (a2) [below of=a3] {$ae$};
\node (a1) [below of=a2] {$ab$};

\node (b4) [right of=a4] {$be$};
\node (b3) [below of=b4] {$bd$};
\node (b2) [below of=b3] {$bc$};
\node (b1) [below of=b2] {$ba$};

\node (c4) [right of=b4] {$ca$};
\node (c3) [below of=c4] {$cb$};
\node (c2) [below of=c3] {$ce$};
\node (c1) [below of=c2] {$cd$};

\node (d4) [right of=c4] {$db$};
\node (d3) [below of=d4] {$da$};
\node (d2) [below of=d3] {$de$};
\node (d1) [below of=d2] {$dc$};

\node (e4) [right of=d4] {$eb$};
\node (e3) [below of=e4] {$ed$};
\node (e2) [below of=e3] {$ea$};
\node (e1) [below of=e2] {$ec$};

\draw [arrow] (a1) -- (a2);
\draw [arrow] (a2) -- (a3);
\draw [arrow] (a3) -- (a4);

\draw [arrow] (b1) -- (b2);
\draw [arrow] (b2) -- (b3);
\draw [arrow] (b3) -- (b4);

\draw [arrow] (c1) -- (c2);
\draw [arrow] (c2) -- (c3);
\draw [arrow] (c3) -- (c4);

\draw [arrow] (d1) -- (d2);
\draw [arrow] (d2) -- (d3);
\draw [arrow] (d3) -- (d4);

\draw [arrow] (e1) -- (e2);
\draw [arrow] (e2) -- (e3);
\draw [arrow] (e3) -- (e4);

\begin{comment}
\draw [arrow] (a1) -- node[anchor=east] {$\preccurlyeq_a$} (a2);
\draw [arrow] (a2) -- node[anchor=east] {$\preccurlyeq_a$} (a3);
\draw [arrow] (a3) -- node[anchor=east] {$\preccurlyeq_a$} (a4);

\draw [arrow] (b1) -- node[anchor=east] {$\preccurlyeq_b$} (b2);
\draw [arrow] (b2) -- node[anchor=east] {$\preccurlyeq_b$} (b3);
\draw [arrow] (b3) -- node[anchor=east] {$\preccurlyeq_b$} (b4);

\draw [arrow] (c1) -- node[anchor=east] {$\preccurlyeq_c$} (c2);
\draw [arrow] (c2) -- node[anchor=east] {$\preccurlyeq_c$} (c3);
\draw [arrow] (c3) -- node[anchor=east] {$\preccurlyeq_c$} (c4);

\draw [arrow] (d1) -- node[anchor=east] {$\preccurlyeq_d$} (d2);
\draw [arrow] (d2) -- node[anchor=east] {$\preccurlyeq_d$} (d3);
\draw [arrow] (d3) -- node[anchor=east] {$\preccurlyeq_d$} (d4);

\draw [arrow] (e1) -- node[anchor=east] {$\preccurlyeq_e$} (e2);
\draw [arrow] (e2) -- node[anchor=east] {$\preccurlyeq_e$} (e3);
\draw [arrow] (e3) -- node[anchor=east] {$\preccurlyeq_e$} (e4);
\end{comment}

\node (ac) [right of=e4,xshift=2.5cm,yshift=.5cm] {$ac$};
\node (ad) [below of=ac,yshift=-.4cm] {$ad$};
\node (de) [below of=ad,yshift=.4cm] {$de$};
\node (ae) [below of=de,yshift=.4cm] {$ae$};
\node (ab) [below of=ae] {$ab$};

\node (be) [right of=ac,xshift=1cm,yshift=.8cm] {$be$};
\node (bd) [below of=be] {$bd$};
\node (bc) [below of=bd,yshift=-.65cm] {$bc$};
\node (ce) [below of=bc,yshift=-.65cm] {$ce$};
\node (cd) [below of=ce] {$cd$};

\draw [arrow] (ab) -- (ae);
\draw [arrow] (ae) -- (de);
\draw [arrow] (de) -- (ad);
\draw [arrow] (ad) -- (ac);

\draw [arrow] (cd) -- (ce);
\draw [arrow] (ce) -- (bc);
\draw [arrow] (bc) -- (bd);
\draw [arrow] (bd) -- (be);

\draw [arrow] (ab) -- (bc);
\draw [arrow] (ce) -- (ae);
\draw [arrow] (bc) -- (ac);
\draw [arrow] (ad) -- (bd);

\end{tikzpicture}

\end{center}
\end{figure}
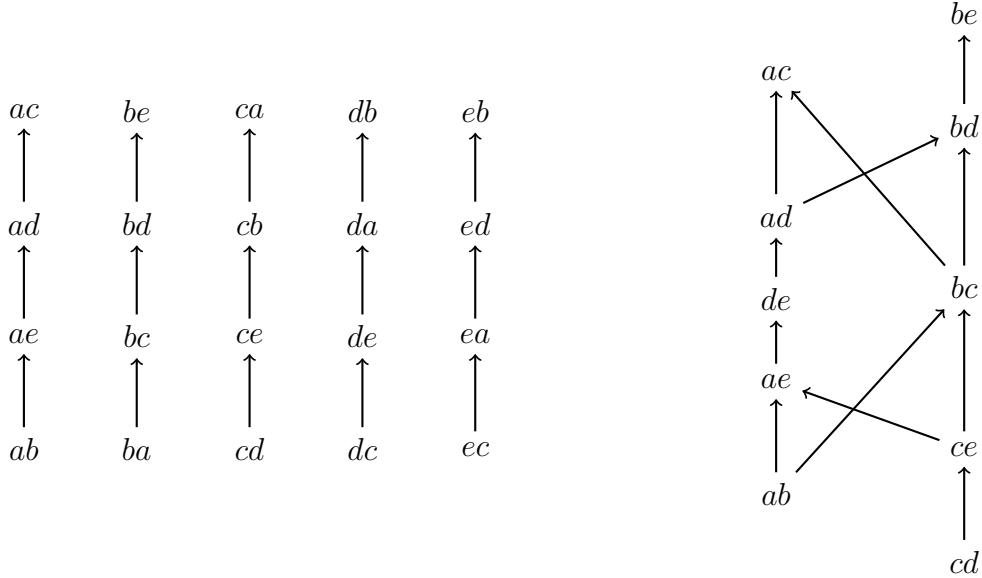

\begin{ex}[Non-concordant ranking system]
\label{NonConcRS}
The (unique up to isomorphism) smallest non-concordant ranking system $\mathcal{N}$ has $S=\{a,b,c\}$ and $(1,2)=(r_a(b),r_a(c))=(r_b(c),r_b(a))=(r_c(a),r_c(b))$, and thus $ab \preccurlyeq_a ac$, $bc \preccurlyeq_b ba$, and $ca \preccurlyeq_c cb$. Any relation $R$ containing $R_\mathcal{N}$ satisfies $ab\;R\;ac\;R\;bc\;R\;ab$, and thus cannot be both transitive and antisymmetric.
\end{ex}

``Almost no'' ranking systems are concordant, in the sense that if a set $S$ is made into a ranking system uniformly randomly as in Section \ref{s:uniformrandomRS}, then the probability that it is concordant tends to 0 as $|S| \rightarrow \infty$; % it is extremely unlikely to be concordant. 
for such a uniformly chosen ranking is expected to contain $\binom{|S|}{3}/4$ triples $\{x,y,z\} \subseteq S$ that are isomorphic to the $\mathcal{N}$ of Example \ref{NonConcRS}, and even the lack of such a triple by no means implies concordance.

It turns out that the concordant ranking systems
are precisely the metrizable ones. 

\begin{lem} \label{l:ranks2metric}
\label{lem:ranking}
For a ranking system $\mathcal{S} = (S,\mathbf{r})$, the following are equivalent:
\begin{itemize}
\item[(A)] $\mathcal{S}$ is concordant;
\item[(B)] $\mathcal{S}$ is metrizable.

%\item[(B)] There exists $d: S \times S \rightarrow [0, \infty)$ such that  
%\begin{itemize}
%    \item[(i)] $(S,d)$ is a metric space, and
%    \item[(ii)] for each distinct $x,y,z \in S$, $d(x,y) < d(x,z)$ iff $x$ prefers $y$ to $z$.
%\end{itemize} 
\end{itemize}
\end{lem}
\textbf{Remark:} Lemma \ref{l:ranks2metric} is an easier result than
 \cite[Theorem 4] {kle14} on \textit{ordinal embedding}, for which
 a linear order on all $\binom{n}{2}$ inter-point distances is the input, and a Euclidean embedding is the output.

\begin{proof}
(B) $\Rightarrow$ (A). Suppose $\rho$ is
 a metric on $S$ inducing $\mathbf{r}$.
Since $\mathbf{r}$ has no ties, we are assured that 
$\rho(x,y) < \rho(x,w)$ if $r_x(y) < r_x(w)$.
%Let us make a temporary stronger assumption that all $\binom{n}{2}$
%inter-point distances under $\rho$ are distinct, so that
%$\rho(x,y) \neq \rho(z,w)$ if $|\{x, y, z, w\}|=4$.
Define a binary relation $\preccurlyeq$ on $\binom{S}{2}$ by
 $xy \preccurlyeq zw$ iff either $xy = zw$ or 
% $xy \prec zw$, where $xy \prec zw$ means that
 $\rho(x,y) < \rho(z,w)$ (well-defined by the symmetry of $\rho$). 
 
 Evidently this is a partial order since it is reflexive ($xy \preccurlyeq xy$),
 and inherits antisymmetry and transitivity from the usual order on $[0,\infty)$, the codomain of $\rho$. 
 Now for any $x \in S$, and $y,z \in S \setminus \{x\}$, we have
\[ xy \preccurlyeq_x xz \quad \Longleftrightarrow \quad
r_x(y) \leq r_x(z) \quad \Longleftrightarrow \quad 
\rho(x,y) \leq \rho(x,z) \quad \Longleftrightarrow \quad
xy \preccurlyeq xz, \]
using the fact that $\rho$ induces $\mathbf{r}$ for the middle equivalence. Thus $\preccurlyeq$ extends each $\preccurlyeq_x$, so it extends $R_\mathcal{S}$.

\begin{comment}
It remains to deal with the case where there
exist 4-element sets $\{x, y, z, w\}$ for which $\rho(x,y) = \rho(z,w)$.
Choose an arbitrary bijection $\chi:S \to [n]$, and change the
definition of $xy \prec zw$ to mean: $\rho(x,y) < \rho(z,w)$, or else
$\rho(x,y) = \rho(z,w)$, $|\{x, y, z, w\}|=4$, and
$\min \{\chi(x), \chi(y)\} < \min \{\chi(z), \chi(w)\}$.
Effectively we are using the bijection as a consistent tie-breaker.
Transitivity is preserved because if $uv \prec xy$ and $xy \prec zw$,
with $\rho(u,u) = \rho(x,y) = \rho(z,w)$, then 
$\min \{\chi(u), \chi(v)\} < \min \{\chi(z), \chi(w)\}$, and so
$uv \prec zw$. The argument continues as before, to show that
$\preccurlyeq$ extends $R_\mathcal{S}$.
\end{comment}

(A) $\Rightarrow$ (B). Set $N:= \binom{n}{2}$ and let $\ell^N_\infty$ denote the $N$-dimensional real Banach space equipped with the sup (max) norm $\|\cdot\|_\infty$. Our goal is to exhibit an embedding $f: S \rightarrow \ell^N_\infty$ satisfying
\begin{align}
    \label{eq:LemConcMetrGoal}
    \text{for each distinct $x,y,z \in S$, $\|f(x)-f(y)\|_\infty < \|f(x)-f(z)\|_\infty$ iff $x$ prefers $y$ to $z$.}
\end{align}

\begin{comment}
into a known metric space $(M, \delta)$, where 
\begin{align*}
\text{(ii') for each distinct $x,y,z \in S$, $\delta(f(x),f(y)) < \delta(f(x),f(z))$ iff $x$ prefers $y$ to $z$.}
\end{align*}
This will suffice: set $d(\cdot,\cdot) = \delta(f(\cdot),f(\cdot))$. We use $(M,\delta) = \mathbb{R}^{\binom{S}{2}}$ (the space of functions from $\binom{S}{2}$ to $\mathbb{R}$) with metric induced by the sup (max) norm $\|\cdot\|_\infty$.
\end{comment}

Let $\preccurlyeq$ be any linear order on $\binom{S}{2}$ that 
extends\footnote{
Here a \emph{linear extension} $\preccurlyeq_l$ of a partial order $\preccurlyeq_p$ is a linear order on the same ground set as $\preccurlyeq_p$ satisfying $x \preccurlyeq_l y$ whenever $x \preccurlyeq_p y$. Every partial order has a linear extension.}
$\preccurlyeq_\mathcal{S}$. Define $\sigma:\binom{S}{2}\rightarrow [N]:=\{1, 2, \ldots, N\}$ by  
\[ \sigma(xy) = \left|\left\{ zw \in \tbinom{S}{2} \mid zw\preccurlyeq xy\right\}\right| \]
($\sigma$ enumerates the elements of $\binom{S}{2}$ from ``bottom up'' according to $\preccurlyeq$). Let $\preccurlyeq'$ be an arbitrary linear order on $S$ (needed for technical convenience only). 
%$\pi:S \rightarrow [|S|]$ be an arbitrary bijection. 
For $x,y,z \in S$ with $y \neq z$ and $y \preccurlyeq' z$, set the $\sigma(yz)$th coordinate of $f(x)$ as:
\[ f(x)_{\sigma(yz)} = 
\begin{cases}
\;\;\; 1 + \sigma(yz)/N \quad &\text{if } x=y \\
-1 -\sigma(yz)/N \quad &\text{if } x=z \\
\;\;\; 0 \quad &\text{otherwise.}
\end{cases} \]
(%Recall that for each $x \in S$, $f(x): \binom{S}{2} \rightarrow \mathbb{R}$.
To see ``what $f$ does,'' it may be helpful to read Example \ref{ex:embedding} before continuing.)

We must check \eqref{eq:LemConcMetrGoal}. Observe that for all $x \in S$ and $zw\in \binom{S}{2}$, each coordinate satisfies $|f(x)_{\sigma(zw)}| \leq 2$. Moreover, for distinct $x,y \in S$, the only coordinate on which $f(x)$ and $f(y)$ are both nonzero is the $\sigma(xy)$th,
%$\text{support}(f(x)) \cap \text{support}(f(y)) = \{xy\}$, 
and 
\[ |f(x)_{\sigma(xy)} - f(y)_{\sigma(xy)}| = 2 + 2\sigma(xy)/N, \]
which is strictly greater than 2. Thus 
\[
\|f(x) - f(y)\|_\infty = 2+2\sigma(xy)/N.
\] This completes the proof: $\preccurlyeq$ extends $\preccurlyeq_x$, so $x$ prefers $y$ to $z$ iff $xy \preccurlyeq xz$ iff $\sigma(xy) < \sigma(xz)$ iff $\|f(x) - f(y)\|_\infty < \|f(x) - f(z)\|_\infty$.
\end{proof}

\begin{ex}[Embedding a CRS into $\ell_{\infty}^N$]
\label{ex:embedding}
%We continue Example \ref{ex:ConcRS} from above. 
Let $\mathcal{S}$ be as in Example \ref{ex:ConcRS}, and suppose we want to embed $S$ into $\ell_\infty^{10}$ as described in the proof of Lemma \ref{lem:ranking}. We first pick a linear extension $\preccurlyeq$ %on $\binom{S}{2}$ that extends 
of $\preccurlyeq_\mathcal{S}$, say the one shown in the top row of the following table. (The reader can verify by comparing to the Hasse diagram for $\preccurlyeq_\mathcal{S}$ that $\preccurlyeq$ is in fact such an extension.)
%Notice that $\preccurlyeq$ is one of the many linear orders extending $\preccurlyeq_\mathcal{S}$. 
Next we pick an arbitrary linear order $\preccurlyeq'$ on $S$,
which will determine the signs in the matrix below; say $\preccurlyeq'$ orders $S$ alphabetically. % $\pi(a)=1, \ldots, \pi(e)=5$. 
Then the function $f$ constructed in the proof of Lemma \ref{lem:ranking} is given in the following table, i.e.\ the value at row $x$, column $yz$ is $f(x)_{\sigma(yz)}$ (where a blank entry means 0).
\begin{comment}
\begin{align*}
&\hspace{5pt}\{a,b\} \preccurlyeq \{c,d\} \preccurlyeq \{c,e\} \preccurlyeq\{a,e\} \preccurlyeq\{b,c\} \preccurlyeq\{d,e\} \preccurlyeq\{a,d\} \preccurlyeq\{a,c\} \preccurlyeq\{b,d\} \preccurlyeq \{b,e\} \\
a\;\;\; &\hspace{14pt}1.1 \hspace{115pt}1.4 \hspace{119pt}1.7 \hspace{30pt}1.8 \\
b\;\;\; &-1.1 \hspace{159pt}1.5 \hspace{164pt}1.9 \hspace{33pt} 2 \\
c\;\;\; &\hspace{55pt}1.2 \hspace{30pt}1.3 \hspace{59pt}-1.5 \hspace{105pt}-1.8 \\
d\;\;\; &\hspace{40pt}-1.2 \hspace{163pt}1.6 \hspace{15pt}-1.7 \hspace{61pt}-1.9 \\
e\;\;\; &\hspace{86pt}-1.3 \hspace{14pt}-1.4 \hspace{58pt}-1.6 \hspace{155pt}-2
\end{align*}
\end{comment}
\[ 
\begin{array}{*{11}r}
 &ab&\preccurlyeq cd&\preccurlyeq ce&\preccurlyeq ae&\preccurlyeq bc&\preccurlyeq de&\preccurlyeq ad&\preccurlyeq ac&\preccurlyeq bd&\preccurlyeq be\\
a& 1.1&    &    & 1.4&    &    & 1.7& 1.8&    &    \\
b&-1.1&    &    &    & 1.5&    &    &    & 1.9& 2  \\
c&    & 1.2& 1.3&    &-1.5&    &    &-1.8&    &    \\
d&    &-1.2&    &    &    & 1.6&-1.7&    &-1.9&    \\
e&    &    &-1.3&-1.4&    &-1.6&    &    &    &-2  
\end{array}
\]
Note: each column $xy$ has nonzero entries exactly at rows $x,y$; all column sums are 0; and the absolute value of every nonzero entry is in $(1,2]$. All of this implies that the $\ell_\infty$-difference between rows $x$ and $y$ is realized uniquely at column $xy$. This, along with the fact that the absolute values of the nonzero entries increase from left to right (in $\preccurlyeq$ order), implies \eqref{eq:LemConcMetrGoal}. 
\end{ex}

\subsection{Sampling from the CRSes}
We saw in Section \ref{s:uniformrandomRS} that FOF-based algorithms fail for generic ranking systems. We aim to show that FOF-based algorithms fail even for generic \emph{concordant} ranking systems.
%We aim to show that this remains true even when we restrict to \emph{concordant} ranking systems.

Given the concordancy restriction, the word ``generic'' presents a difficulty. In view of the comment after Example \ref{NonConcRS}, a uniformly random ranking system will almost certainly fail to be concordant. 
%it is easy to generate a generic ranking system on a set $S$---simply pick one uniformly. But 

%Since only the concordant ranking systems correspond to metric spaces, and since metric spaces are the context in which we are studying the performance of the NND algorithm, we would like a way to generate a generic \emph{concordant} ranking system. 

%This is subtler than it may seem. 
We know of no feasible method to sample uniformly\footnote{
Given a concordant ranking system, there is a Markov chain
whose stationary state gives a uniform random linear extension to $\binom{S}{2}$: see Huber \cite{hub}.
} from the CRSes on a set $S$. As there are $(|S|-1)!^{|S|}$ ranking systems, simply listing them and marking the concordant ones is impractical whenever $|S|>6$ (observe $6!^7 \approx 10^{22}$). 
%This is to be interpreted mostly as a theoretical statement, because even writing down a ranking system on $n$ elements is a $\Theta(n^2)$ operation---computationally infeasible for $n$ on the scale of a realistic dataset. (For small $n$, the low fraction of ranking systems that are concordant rules out drawing a uniform ranking system repeatedly until a concordant one is found.) What we would like is some well-understood finite probability space $(\Omega_n,\mu_n)$ and function $f_n$ on $\Omega_n$ such that $f_n(\omega)$ is a uniformly chosen concordant ranking system on $n$ elements when $\omega \in \Omega_n$ is chosen according to $\mu_n$. We lack even this, because we do not know how to list
%%---or even count---
%the concordant ranking systems (in a human timeframe).

Here is an alternative approach. Given a linear order $\preccurlyeq$ on $\binom{S}{2}$, define, for each distinct $x,y \in S$, % and $y \in S \setminus \{x\}$, define % $r_x$ as follows: 
\begin{align}
\label{eq:ranksFromLinearOrder}
r_x(y) = \left|\left\{z \in S \setminus \{x\} \;:\; xz\preccurlyeq xy\right\}\right|.
\end{align}
Equivalently, define $\preccurlyeq_x$ as the restriction of $\preccurlyeq$ to $S_x$. This corresponds to reading row $x$ in Example \ref{ex:embedding} and listing the column headings of its nonzero entries left to right.

Evidently the resulting ranking system $(S,(r_x)_{x\in S})$ is concordant, because there exists a partial order extending $\cup_{x\in S}\preccurlyeq_x$, namely $\preccurlyeq$.

\begin{comment}
It is easy to generate a uniform linear order $\preccurlyeq$ on $\binom{S}{2}$. To obtain a concordant ranking system $\mathcal{S}$ on $S$, simply declare that $\preccurlyeq$ extends $\preccurlyeq_\mathcal{S}$. Having done so, we can read off $r_x$ from $\preccurlyeq$ for each $x\in S$: for all $y \in S\setminus \{x\}$,
\begin{align}
%\label{eq:ranksFromLinearOrder}
r_x(y) = \left|\left\{z \in S \setminus \{x\} \;:\; xz\preccurlyeq xy\right\}\right|;
\end{align}
equivalently, $\preccurlyeq_x$ is the restriction of $\preccurlyeq$ to $S_x$. This corresponds to reading row $x$ in Example \ref{ex:embedding} and listing the column headings of its nonzero entries left to right.
\end{comment}

\begin{defn}
If $\preccurlyeq$ is a linear order on $\tbinom{S}{2}$ for some finite set $S$, let $\varphi(\preccurlyeq)$ denote the CRS on $S$ defined by \eqref{eq:ranksFromLinearOrder}:
\begin{align}
\label{eq:phimap}
\{\text{\emph{linear orders on $\tbinom{S}{2}\} \quad \xlongrightarrow{\varphi} \quad \{$CRSes on $S\}$. }}
\end{align}
\end{defn}

Observe that the preimages under $\varphi$ of a CRS $\mathcal{S}$ are precisely the linear extensions of $\preccurlyeq_\mathcal{S}$. (This is tautological: by construction $\preccurlyeq$ extends each $\preccurlyeq_x$ of $\varphi(\preccurlyeq)$, %$\cup_{x\in S}\preccurlyeq_x$, 
and by definition $\preccurlyeq_{\varphi(\preccurlyeq)}$ is the \emph{minimal} partial order that does so,
%extending $\cup_{x\in S}\preccurlyeq_x$, 
so $\preccurlyeq$ must extend $\preccurlyeq_{\varphi(\preccurlyeq)}$.) Since every partial order has a linear extension, $\varphi$ is surjective. However it is not injective. For example the $\preccurlyeq_\mathcal{S}$ of Example \ref{ex:ConcRS} has many linear extensions, i.e.\ preimages. %---if it were, then $\varphi(\preccurlyeq)$ would be a uniformly random CRS
We explore this non-injectivity at length in Section \ref{sec:varphi}.

Use the term \emph{generic} CRS on $S$ to refer to the image
$\varphi(\preccurlyeq)$ of a uniformly random 
linear order $\preccurlyeq$ on $\binom{S}{2}$. This corresponds to choosing a metric on $S$ for which the inter-point distances are randomized. Let us reiterate that a \emph{generic} CRS is far from
being a uniformly random CRS.

%%%%%%%%%%%%%%%%%%%%%%%%%%%%%%%%%%%%%%%%%%%%%%%%%%%%%%%%%%%%
\subsection{NND fails for a generic CRS}

As in Section \ref{s:uniformrandomRS} we analyze scheduled pointwise 
version NND of
Section \ref{sec:pointwise},
choosing this variant only for convenience.

 \begin{prop}
 \label{prop:KNNfails}
 %Fix $1 \leq K < \sqrt{n}$. 
 Let $\preccurlyeq$ be a uniformly random linear order on $\binom{S}{2}$. Running scheduled pointwise NND on $\varphi(\preccurlyeq)$,
 %Assume that a concordant ranking system $r$ is sampled uniformly, as described in Lemma \ref{l:sampleranking}.
 for each $x\in S$ the expected number of rounds needed
 before $x$'s friend set contains at least $\lceil K/2\rceil$ elements of  $B_x:=\{y \in X \setminus \{x\} \;:\; r_x(y) \leq K\}$
 is at least about $n^2/(2K^2)$.
 \end{prop}

\begin{comment}
\textbf{Remark:} Since each step in the algorithm
requires about $n K^2$ rank computations, the total (expected) work
is $\Theta(n^2)$, without even a certainty of finding all the
top $K$ ranked elements. In this sense,
$K$-nearest neighbor descent underperforms simple 
exhaustive computation of $r_x(y)$ for all pairs $x, y$,
when applied to a $\varphi$-selected concordant ranking system.
This discussion is not purely theoretical. Section \ref{s:lcss}
gives a similar example from string matching.
\end{comment}

Proposition \ref{prop:KNNfails} follows easily from the following elementary fact, whose proof we omit.

\begin{lem}
\label{lem:permIndep}
Let $\preccurlyeq$ be a uniformly chosen linear order on a finite set $X$, and let $A,B \subseteq X$ be disjoint. Then the random variables $\preccurlyeq\hspace{-4pt}\vert_A$ (that is, $\preccurlyeq$ restricted to $A$) and $\preccurlyeq\hspace{-4pt}|_B$ are independent.
\end{lem}

\begin{proof}[Proof of Proposition \ref{prop:KNNfails}]
Recall $S_x := \{ xy : y \in S \setminus \{x\}\}$. Fix $x \in X$, and
let $A=\binom{S}{2} \setminus S_x$, i.e. point pairs not including $x$.
Take
\[
\mathcal{S'}=(S\setminus \{x\},\mathbf{r'}) = 
\varphi(\preccurlyeq\hspace{-4pt}|_A).
\]
Then Lemma \ref{lem:permIndep} implies that $r_x$, that is $x$'s ranking in $\varphi(\preccurlyeq)$, is independent of $\mathbf{r'}$, that is anyone else's ranking of any third party. 
%Suppose $j \in F_{t-1}(i)$, and $k \in F_{t-1}(j)$.
%The distance $d(i, k)$ is given by the random variable 
%$2 U_{i,k}$,
%which is independent of $U_{i,j}$ and $U_{j,k}$.
In other words, the fact that $y$ is at some time a friend of $x$ does not
make $y$'s friend $z$ any better for $x$
than a uniformly selected vertex.
The rest of the proof exactly follows Section \ref{s:uniformrandomRS}:
in each of $x$'s friend list requests she discovers at most $K^2$ new points, which are from her perspective uniformly random. 
Thus $x$ expects to have to meet about $n/2$ points, over 
at least $n/(2K^2)$ relevant rounds, to find $\lceil K/2 \rceil$ points of $B_x$.
\end{proof}

\subsection{CRS perspective on longest common substring}
%\textbf{Concordant ranking system perspective on the longest-common-substring example:}
Consider three distinct random strings $\mathbf{x,y,z}$ generated as in Section \ref{s:lcss}.
Suppose $\mathbf{a}$, $\mathbf{b}$, and $\mathbf{c}$
are the longest common substrings between $\mathbf{x}$ and $\mathbf{y}$,
$\mathbf{y}$ and $\mathbf{z}$, and $\mathbf{z}$ and $\mathbf{x}$ respectively.
%During the course of any NND %$K$-nearest neighbor descent computation,
The lengths of $\mathbf{a, b, c}$ will not exceed $q_1 \ll m$.
Since $\mathbf{a}$ and $\mathbf{b}$ are so short, and likely
to occur in different parts of $\mathbf{y}$,
their existence does not help to supply a long $\mathbf{c}$.
In other words, for $q = 1, 2, \ldots, q_1$
\[
\Pr[M(\mathbf{x,z}) \geq q \mid 
\min \{ M\mathbf{(x,y}), M(\mathbf{y,z})\} \geq q]
\approx \Pr[M(\mathbf{x,z}) \geq q ].
\]
This is essentially the situation when we choose a generic CRS with random inter-point distances, 
which leads to quadratic run time for NND. %$K$-nearest neighbor descent.

\subsection{Digression: properties of the map $\varphi$}
\label{sec:varphi}

Here we begin to explore the many mysteries of the map $\varphi$. We believe this is fertile ground for future research.

Fix a set $S$ of size $n$. Let $L_n$ denote the
set of $\binom{n}{2}!$ linear orders on $\binom{S}{2}$,
and let $R_n$ denote the set of CRSes on $S$.
Our vantage point for insight into
the map (\ref{eq:phimap}) $\varphi:L_n \to R_n$ is the following graphical model.

Consider the graph $G$ on vertex set $L_n$, taking
$\preccurlyeq$ and $\preccurlyeq'$ to be adjacent in $G$ iff $\preccurlyeq'$ agrees with $\preccurlyeq$
after swapping the order of consecutive items $xy$ and $zw$; here the intersection of $\{x, y\}$ and $\{z, w\}$ could be non-empty.
Call $G$ the \textbf{equivalent metrics graph}. Any linear order has $d:=\binom{n}{2}-1$ pairs
of consecutive items, so $G$ is $d$-regular. Any linear order can be rearranged into any other via some sequence of swaps of consecutive items, so $G$ is connected.

Color the edges white or black as follows.
An edge between $\preccurlyeq$ and $\preccurlyeq'$
is colored white if \linebreak $\varphi(\preccurlyeq) = \varphi(\preccurlyeq')$, i.e.\
if these two linear orders map to the same ranking
system, and colored black otherwise.

\begin{prop}
\label{prop:swap}
An edge of $G$ joining $\preccurlyeq$ and $\preccurlyeq'$ is white iff the consecutive items $xy,zw$ that are swapped between $\preccurlyeq$ and $\preccurlyeq'$ are disjoint.
%Let $\preccurlyeq$ be a linear order on $\tbinom{S}{2}$ and let $xy \neq zw$ be consecutive in $\preccurlyeq$. Let $\preccurlyeq'$ be the linear order on $\tbinom{S}{2}$ that agrees with $\preccurlyeq$ except for swapping the order of $xy$ and $zw$. Then $\varphi(\preccurlyeq) = \varphi(\preccurlyeq')$ iff $xy \cap zw = \varnothing$.
\end{prop}

\begin{proof}
Simply observe that if $\{x, y\}\cap \{z, w\} =\varnothing$, then for any $u \neq v \in S$, the computation of $r_u(v)$ in \eqref{eq:ranksFromLinearOrder} does not change when $\preccurlyeq$ is replaced by $\preccurlyeq'$. If on the other hand $x=w$  (say), then $r_x(y)$ and $r_x(z)$ swap when $\preccurlyeq$ is replaced by $\preccurlyeq'$.
\end{proof}

\begin{prop}
\label{prop:pathofswaps}
Let $\preccurlyeq$ and $\preccurlyeq'$ be linear orders on $\binom{S}{2}$. Then $\varphi(\preccurlyeq)=\varphi(\preccurlyeq')$ iff there is a path of white edges joining $\preccurlyeq$ and $\preccurlyeq'$ in $G$.
%sequence of swaps of consecutive, disjoint elements that transforms $\preccurlyeq$ into $\preccurlyeq'$.
\end{prop}

\begin{proof}
Suppose there is such a path of white edges. Proposition \ref{prop:swap} implies that all the associated swaps are disjoint,
and so $\varphi(\preccurlyeq)=\varphi(\preccurlyeq')$.
Conversely, suppose $\varphi(\preccurlyeq)=\varphi(\preccurlyeq')$.
Let $N=\binom{n}{2}$ and consider the ``selection sort'' algorithm to rearrange the list of pairs $L=(A_1 \preccurlyeq A_2 \preccurlyeq \cdots \preccurlyeq A_N)$ into the list $L'=(A'_1 \preccurlyeq' A'_2 \preccurlyeq' \cdots \preccurlyeq' A'_N)$:
\begin{itemize}%[leftmargin=.2in]
    \item For each $i=1,2,\ldots,N$ do:
    \begin{itemize}
        \item Find $A'_i$ in $L$, say at position $j \geq i$
        \item Move $A'_i$ down to position $i$ via $j-i$ swaps involving $A'_i$ and the item just below.
    \end{itemize}
\end{itemize}
We claim there are no ``bad swaps'' (swaps of nondisjoint items). To see this, observe that in the course of the algorithm, for each $i,j$, items $A_i,A_j$ are swapped at most once. Thus if there were a bad swap, say with $A_i \cap A_j = \{x\}$, then we would have $\varphi(\preccurlyeq) \neq \varphi(\preccurlyeq')$ as witnessed by $r_x$.
\end{proof}

%To reduce $N$ by 2, let $t_0$ be the first index of a bad swap in $\sigma$, say swapping $xy$ and $xz$, and let $t_1$ be the next index of a swap of $xy$ and $xz$. Delete from $\sigma$ all swaps between $t_0$ and $t_1$, inclusive, that involve $xy$ or $xz$ (or both). It is easy to see that: (a) every swap in the new sequence involves consecutive elements, and (b) the new sequence transforms $\preccurlyeq$ to $\preccurlyeq'$.

\begin{cor}
The subgraph of $G$ consisting of the
white edges, which we call the \textbf{\emph{white graph}}, decomposes into %graph 
components, each of which
consists of the inverse image under $\varphi$ of a distinct CRS.
\end{cor}
%A little about this component structure is revealed by the following Lemma, and more by Proposition \ref{l:1stmoment}.

The components of the white graph vary dramatically in size, from as small as 1 to as large as $(n/2)!^{n-1} \approx \exp(n^2\log(n)/2)$ or larger, as the next examples show. To present these examples we adopt the following practical convention. For small $n$, we represent a linear order $\preccurlyeq$ %visually 
as a matrix where rows are indexed by $S$ and each column is the indicator of an element of $\binom{S}{2}$. We agree that 
the columns are arranged in $\preccurlyeq$-increasing order left to right (as in Example \ref{ex:embedding}), and we leave blanks for zeros.

\begin{ex}[CRS with a unique $\varphi$-preimage]
\label{ex:RankingSystemUnique}
We return to the powers of 2 example of Section \ref{sec:powersof2}. Let $S=\{1,2,4,8,16,32\}$ and let $\preccurlyeq$ order the pairs of points of $S$ by their Euclidean distance in $\R$. If row $i$ corresponds to $2^i$, 
and columns are indexed by the vector
\[
\begin{smallmatrix}
 \{1,2\} & \{2,4\} & \{1,4\} & \{4,8\} & \{2,8\} &
  \{1,8\} & \{8,16\} & \{4,16\} & \{2,16\} & \{1,16\} &
   \{16,32\} & \{8,32\} & \{4,32\} & \{2,32\} & \{1,32\}\\
\end{smallmatrix}
\]

then $\preccurlyeq$ is represented by the matrix
\begin{align}
\label{eq:CRSuniquePreimage}
\left[
\begin{array}{*{15}c}
\cdashline{3-3} \cdashline{5-6} \cdashline{8-10} \cdashline{12-15}
1& &\multicolumn{1}{:c:}{1}& &\multicolumn{1}{:c}{ }&\multicolumn{1}{c:}{1}& &\multicolumn{1}{:c}{ }& &\multicolumn{1}{c:}{1}& &\multicolumn{1}{:c}{ }& & &\multicolumn{1}{c:}{1}\\
\cdashline{3-3}
1&1& & &\multicolumn{1}{:c}{1}&\multicolumn{1}{c:}{ }& &\multicolumn{1}{:c}{ }&1&\multicolumn{1}{c:}{ }& &\multicolumn{1}{:c}{ }& &1&\multicolumn{1}{c:}{ }\\ 
\cdashline{5-6}
 &1&1&1& & & &\multicolumn{1}{:c}{1}& &\multicolumn{1}{c:}{ }& &\multicolumn{1}{:c}{ }&1& &\multicolumn{1}{c:}{ }\\
\cdashline{8-10}
 & & &1&1&1&1& & & & &\multicolumn{1}{:c}{1}& & &\multicolumn{1}{c:}{ }\\
\cdashline{12-15}
 & & & & & &1&1&1&1&1& & & & \\ 
 & & & & & & & & & &1&1&1&1&1\\ 
\end{array}
\right].
\end{align}
In view of Proposition \ref{prop:swap}, since every pair of consecutive columns shares a nonzero row, every edge of $G$ incident to $\preccurlyeq$ is black. Thus $\preccurlyeq$ is an isolated point of the white graph. % In other words $\preccurlyeq$ is the unique $\varphi$-preimage of the ``Powers of 2'' CRS $\varphi(\preccurlyeq)$. 
\end{ex}

\begin{comment}
\[ \left[
\begin{array}{*{15}c}
1&1&1&1&1& & & & & & & & & & \\
 & & & &1&1&1&1&1& & & & & & \\ 
 & & &1& & & & &1&1&1&1& & & \\
 & &1& & & & &1& & & &1&1&1& \\
 &1& & & & &1& & & &1& & &1&1\\
1& & & & &1& & & &1& & &1& &1
\end{array}
\right]. \]

\begin{tabular}{|l::c r|}\hline
A&B&C\\\hdashline[1pt/1pt]
AAA&BBB&CCC\\\cdashline{1-2}[.4pt/1pt]
\multicolumn{2}{|l;{2pt/2pt}}{AB}&C\\\hdashline\hdashline
\end{tabular}
\end{comment}

\begin{comment}
\[ \left[
\begin{array}{*{15}c}
1&1&1&1&1& & & & & & & & & & \\
 & & & &1&1&1&1&1& & & & & & \\ 
\cdashline{1-4}
\multicolumn{1}{:c}{ }& & &\multicolumn{1}{c:}{1}& & & & &1&1&1&1& & & \\
\cdashline{6-8}
\multicolumn{1}{:c}{ }& &1&\multicolumn{1}{c:}{ }& &\multicolumn{1}{:c}{ }& &\multicolumn{1}{c:}{1}& & & &1&1&1& \\
\cdashline{10-11}
\multicolumn{1}{:c}{ }&1& &\multicolumn{1}{c:}{ }& &\multicolumn{1}{:c}{ }&1&\multicolumn{1}{c:}{ }& &\multicolumn{1}{:c}{ }&\multicolumn{1}{c:}{1}& & &1&1\\ 
\cdashline{13-13}
\multicolumn{1}{:c}{1}& & &\multicolumn{1}{c:}{ }& &\multicolumn{1}{:c}{1}& &\multicolumn{1}{c:}{ }& &\multicolumn{1}{:c}{1}&\multicolumn{1}{c:}{ }& &\multicolumn{1}{:c:}{1}& &1 \\ 
\cdashline{1-4} \cdashline{6-8} \cdashline{10-11} \cdashline{13-13}
\end{array}
\right]. \]

\[ \left[
\begin{array}{*{15}c}
1& &1& & &1& & & &1& & & & &1\\
1&1& & &1& & & &1& & & & &1& \\ 
 &1&1&1& & & &1& & & & &1& & \\
 & & &1&1&1&1& & & & &1& & & \\
 & & & & & &1&1&1&1&1& & & & \\ 
 & & & & & & & & & &1&1&1&1&1\\ 
\end{array}
\right]. \]
\end{comment}

By examining the matrix \eqref{eq:CRSuniquePreimage} we can derive a lower bound on the number of isolated points of the white graph. Indeed, if we replace each dashed box in \eqref{eq:CRSuniquePreimage} by any permutation matrix of the same size, the resulting matrix represents another isolated point. % of the white graph. 
This is because it retains from \eqref{eq:CRSuniquePreimage} the property that every pair of consecutive columns shares a nonzero row. Thus the white graph contains at least $\prod_{k=1}^{n-2}k! \approx \exp(n^2\log(n)/2)$ isolated points.

%\textcolor{red}{NEW MATERIAL:} 
For odd $n$, another source of isolated points in the white graph is Eulerian circuits\footnote{Recall that an \emph{Eulerian circuit} in a graph is a closed walk that traverses each edge exactly once.} in the complete graph on $S$. Indeed, if $C=(e_1,e_2,\ldots,e_{N})$ is any such circuit, with $N:= \binom{n}{2}$, then the linear order $e_1 \preccurlyeq e_2 \preccurlyeq \cdots \preccurlyeq e_N$ is an isolated point. This is because since $C$ is a walk, by definition $e_i \cap e_{i+1} \neq \varnothing$ for each $i$. McKay and Robinson \cite{mck} prove an asymptotic formula for the number of Eulerian circuits on $K_n$ in which the dominant term is $\exp(n^2\log(n)/2)$, roughly matching the number of isolated points arising from powers-of-2-type orders as above.

% In fact, $\preccurlyeq$ is but one of a family of linear orders with this property, for it is evident that $\preccurlyeq$ generalizes to a ground set of any size $n$. Moreover, we can freely choose the order of the rows in which the horizontal strings of 1's appear in the matrix, and below(/above) the $k$th string of 1's ($k \in [n-1]$), we can fill in any permutation matrix of dimension $n-k-1$. This yields $n!\prod_{k=1}^{n-1}(n-k-1)! = n\prod_{j=1}^{n-1}j!$ such orders. 

\begin{ex}[CRS with many $\varphi$-primages]
\label{ex:RankingSystemMany}
For even $n$, the set $\binom{S}{2}$ of edges of 
the complete graph on $S$ may be partitioned into 
disjoint perfect matchings\footnote{
For each $j$, each vertex in $S$ is incident to exactly one edge in $P_j$.
} $(P_1,\ldots,P_{n-1})$,
according to
an elementary form of Baranyai's theorem \cite{van}.
Here $|P_j|=n/2$ for every $j$.

 Let $\preccurlyeq$ be any linear order on $\binom{S}{2}$ satisfying $xy\preccurlyeq zw$ whenever $xy\in P_i$, $zw\in P_j$, and $i<j$. Then $\varphi(\preccurlyeq)$ has at least $(n/2)!^{n-1}$ $\varphi$-preimages, because each $P_i$ can be arbitrarily reordered (relative to $\preccurlyeq$) independently of the others, without changing the image under $\varphi$. The following matrix represents such a linear order $\preccurlyeq$ with $n=6$.
 Rows correspond to the six elements of $S$, and columns to the $\binom{6}{2}$
 edges of the complete graph on $S$.
\[ \left[
\begin{array}{*4{ccc:}ccc}
1& & & & &1& &1& &1& & &1& & \\
1& & &1& & & & &1& &1& & & &1\\
 &1& &1& & &1& & &1& & & &1& \\
 &1& & &1& & &1& & & &1& & &1\\
 & &1& &1& &1& & & &1& &1& & \\
 & &1& & &1& & &1& & &1& &1& 
\end{array}
\right] \]
The dashed vertical lines partition the columns into $n-1$ partitions of $S$, each of size $n/2$. The first three columns represent $P_1$, the next three $P_2$, etc. Notice that in this example even the consecutive columns that cross the boundaries of the partitions share no nonzero rows, % no pair of consecutive columns shares a nonzero row, 
meaning that every edge incident to $\preccurlyeq$ in the equivalent metrics graph is white.
\end{ex}

% The upshot of Propositions \ref{ex:RankingSystemUnique} and \ref{ex:RankingSystemMany} is that for a linear order $\preccurlyeq$ on $\binom{S}{2}$ chosen uniformly, the (ranking-system-valued) random variable $\varphi(\preccurlyeq)$ is far from uniform. Nonetheless, in a certain sense it is ``almost uniform,'' as the next lemma shows.

We close this section with two easy calculations intended to give some (minimal) sense of the overall structure of the equivalent metrics graph.

\begin{prop}
The probability $p$ that an edge of $G$
selected uniformly at random is white is
\begin{align*}\label{e:blackedge}
      p:=\frac{\binom{n-2}{2} }{\binom{n}{2} - 1} = 
    %  \frac{n-3}{n+1} = 
1 - \frac{4}{n+1}.
%\sim 1- \frac{4}{n}. 
\end{align*}
\end{prop}

%\textbf{Remark: }
%Examples above show that there exist some
%vertices of $G$ at which all incident edges are
%black, others at which all incident edges are white, and some with a mixture.

\begin{proof}
Pick a uniformly random edge by (a) picking a uniformly random vertex and then (b) picking a uniformly random edge incident to it. At step (b), out of $d$ possible choices of a pair of consecutive elements $xy,zw$, on average $\binom{n-2}{2}$ of these pairs will be disjoint.
%The denominator is $d$, i.e.\ the number of ways to pick two
%consecutive elements out of an ordered list of
%$\binom{n}{2}$ elements.
%The numerator is the expected number of ways 
%to pick these
%two elements, which are 2-sets $xy$ and $zw$, such that they are disjoint,
%fgiven that the list was randomly ordered.
\end{proof}

\begin{prop}[Linear orders per CRS]
\label{l:1stmoment}
Let $H$ count the linear orders on $\binom{S}{2}$ which
lie in the preimage under $\varphi:L_n \to R_n$ of a 
CRS on $S$ chosen uniformly at random. Then
\begin{equation} \label{e:ratio}
\frac{\binom{n}{2}!}{((n-1)!)^n} \;\; < \;\; 
\E[H] = \frac{|L_n|} {|R_n|} \;\; < \;\;
\frac{\binom{n}{2}!}{\prod_{k=1}^{n-2}k!}.
\end{equation}
\end{prop}

\textbf{Remark:} 
Applying Stirling's approximation for the factorials,
the lower bound for $\log{\E[H]}$ has leading term
%$n^2 (1 - \log{(2)}) / 2 \approx 0.3495 \ldots n^2$, %% I'M PRETTY SURE THIS IS OFF BY THE LOCATION OF A PARENTHESIS
$n^2 (1 - \log(2)) / 2 \approx 0.15 n^2$,
whereas the upper bound for the log has leading term
$n^2\log(n) /2$.

\begin{proof}
The equality in \eqref{e:ratio} follows from  %we obtain assertion (\ref{e:ratio}) follows from:
\[  
|R_n| \E[H] 
% = \sum_m m |\{r \in R_n: |\varphi^{-1}(r)| = m\}| 
= \sum_{r \in R_n} |\varphi^{-1}(r)| = |L_n|.
\]
For the inequalities, let $\tilde{R}_n \supset R_n$ denote
the set of all ranking systems (possibly non-concordant) on $S$,
and let $R^1_n \subset R_n$ denote the concordant $r \in R_n$
for which $|\varphi^{-1}(r)|= 1$. It was proved above that 
\[
|R_n| > |R^1_n| \geq \prod_{k=1}^{n-2}k!.
\]
It is immediate from the definitions that
\[
|R_n| < |\tilde{R}_n| = ((n-1)!)^n.
\]
Replacing $|R_n|$ in \eqref{e:ratio} by $|R^1_n|$ and $|\tilde{R}_n|$, the bounds follow.
\end{proof}

%%%%%%%%%%%%%%%%%%%%%%%%%%%%%%%%%%%%%%%%%%%%%%%%%%%%%%%%%%%%%%%%%%
%%%%%%%%%%%%%%%%%%%%%%%%%%%%%%%%%%%%%%%%%%%%%%%%%%%%%%%%%%%%%%%%%%
%%%%%%%%%%%%%%%%%%%%%%%%%%%%%%%%%%%%%%%%%%%%%%%%%%%%%%%%%%%%%%%%%%
\section{Second Neighbor Range Query on a Homogeneous Poisson Process} \label{s:knnpp}

\subsection{Second neighbor range query} \label{s:2nrq-describe}

We are disappointed to report that our attempts to
bound the complexity of pure ranking-based NND implementations
described in Section \ref{s:implement} were not successful.
Instead we shall analyze another FOF-based algorithm, called second neighbor range query. 

Suppose NND were performed on a set $S$ whose elements are distributed uniformly 
in a compact metric space $(\mathcal{X}, \rho)$.
Given the index $t$ of the current round, we should be able to predict roughly the distance $r$ from a typical point to its least-preferred friend. The \emph{second neighbor range query algorithm} (2NRQ) harnesses this intuition.
The principle of 2NRQ,
made precise in Section \ref{s:2nrq}, may be summarized as:
\begin{center}
    \textit{Each vertex allows each pair of its
    neighbors to discover if they are 
    close to each other.}
\end{center}

The basic unit of work in 2NRQ is not a friend barter but a \emph{range query}. A range query is similar to one round of friend set updates, with changes to the steps that determine membership in the new friend set. In a range query, we fix $r>0$ and we retain some but not all proposed arcs
$u \to u'$ satisfying $\rho(u, u') \leq r$, according to a certain acceptance sampling policy. We discard arcs where $\rho(u, u') > r$.
The out-degree of of a vertex in the friend digraph is no longer exactly $K$; the idea is to choose $r$ so that the expected out-degree is $K$. A single range query is illustrated in Figure \ref{f:2nf}. 

\begin{figure}
\caption{
\textbf{Second neighbor range query: }
\textit{Left: A random graph $G$ with 12 edges among 13 vertices embedded in $[-1, 1]^2$.} 
\textit{Right: The result of a second neighbor range query operation: distinct vertices
with a common neighbor in the left graph become adjacent if
their separation does not exceed $0.7$. Eight new edges are shown in red.
For visual comprehension, edges of $G$ are shown as dashed.
Here we gloss over the
subtle acceptance sampling of the 2NRQ update discussed in Section
\ref{s:2nrq}.}
}
\label{f:2nf}
\begin{center}
\scalebox{0.36}{\includegraphics{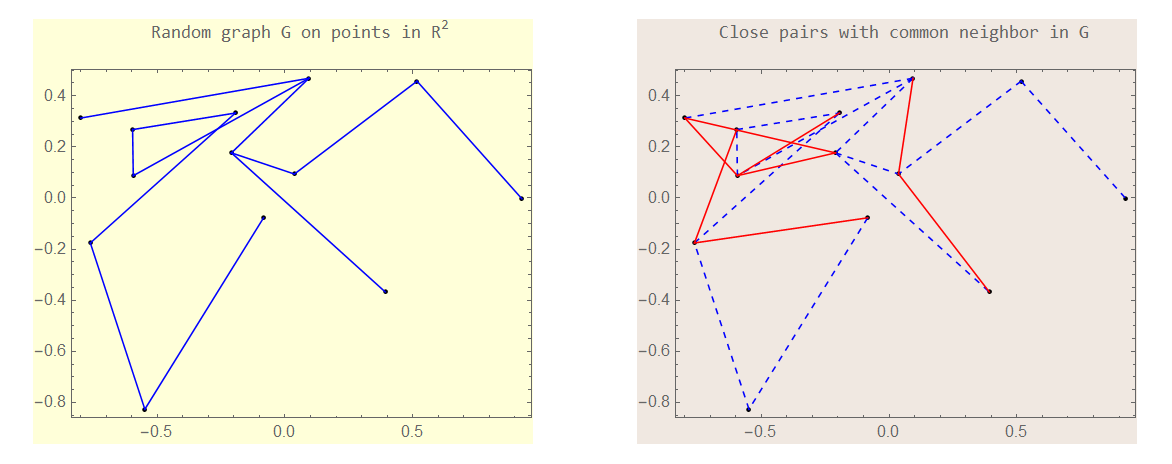}} 
\end{center}
\end{figure}

The 2NRQ algorithm is not proposed
for practical use; it is merely the closest
algorithm to NND
among those whose behavior we can analyze
rigorously, at least for the case where the set $S$
is a realization of a homogeneous Poisson process.
The only case in which explicit parameters for 2NRQ will be derived is
where the metric space is the $d$-dimensional torus
with the $\ell_{\infty}$ norm.

Key differences between 2NRQ and NND are:
\begin{enumerate}
    \item[(a)] 2NRQ requires a metric space, whereas NND does not.
    \item[(b)] 2NRQ determines friends based on distance, whereas NND queries a ranking oracle.
    \item[(c)] 2NRQ allows vertex degrees to vary, whereas in NND they are fixed at $K$.
    \item[(d)] 2NRQ uses undirected edges, whereas NND uses directed arcs.
    \item[(e)] 2NRQ is a theoretical tool for investigating complexity,
    whereas NND is proposed for practical deployment.
\end{enumerate}

\subsection{Compact metric space with invariant measure}\label{s:compactspace}
Take a compact metric space $(\mathcal{X},\rho)$, whose topology
is second countable and Hausdorff\footnote{
Locally compact, second countable Hausdorff spaces are suitable
as a setting for random measures; see Kallenberg \cite[Ch. 10]{kal}.
}, and which supports a positive
finite Borel measure $\lambda$, with the invariance property that
all balls of the same radius have the same measure. 
Without loss of generality, assume that diameter of $\mathcal{X}$ is
$\sup \{\rho(x, y): x, y \in \mathcal{X} \} = 1$.
Denote by $B_r(x)$
the closed ball of radius $r$, centered at $x \in \mathcal{X}$, and denote the ratio of the volume of such a ball to the volume of
the space by
\begin{equation}\label{e:volumeratio}
  h(r):=\frac{\lambda(B_r(x))}{ \lambda(\mathcal{X})}, \quad x \in \mathcal{X}.  
\end{equation}

Take the finite subset $V \subset \mathcal{X}$ to be the support of a Poisson process 
with intensity measure
$n \lambda(\cdot) / \lambda(\mathcal{X})$, as defined in Kallenberg \cite[Ch. 10]{kal}. 
Readers unfamiliar with these terms need only observe that: (1) the size $|V|$
is a Poisson$(n)$ random variable; (2) the number of points of $V$ in any fixed ball of radius $r$ is Poisson$(n h(r))$; and (3) occupancies of disjoint balls are independent.

\subsection{2NRQ graph process}
\subsubsection{Goal}
\label{s:2nrq}
Our goal is to construct a finite sequence $(V, E_t)_{0 \leq t \leq T}$
of undirected random graphs on $V$, with the aim of presenting $(V,E_T)$ %the neighbors
%of each vertex $v \in V$ under $E_T$ 
as an approximate $K$-nearest neighbor
graph. %set. 
In the course of the construction, we will define a
sequence of rates $(\theta_t)$, and a sequence of distances $(r_t)$:
\begin{equation} \label{e:radialparams}
  \frac{K}{n} =  \theta_0 < \theta_1 < \cdots < \theta_T < 1;
\quad
1 = r_0 > r_1 > \cdots > r_T > h^{-1}(K/n),
\end{equation}
coupled together using the identity
\begin{equation}\label{e:samplerate}
    \theta_t = \frac{K}{n h(r_t)}.
\end{equation}
Formulas for $(r_t)$, and a numerical computation, appear in Sections \ref{s:radii-formulas} and \ref{s:explicitradii}. 
The graphs $(E_t)$ will be constructed so as to exhibit
a desired \textbf{sampling property}.
For every $v \in V$:
\begin{center}
    \textit{The neighbors of $v$ under $E_t$ are a random sample
    at rate $\theta_t$ of the elements of}
    \[
    Q_t(v):= (V \setminus \{v\}) \cap B_{r_t}(v).
    \]
\end{center}
In particular, $\rho(v, v') > r_t$ implies $vv' \notin E_t$.

This property holds for $t=0$ by construction:
$E_0$ is obtained by sampling edges of the complete graph on $V$,
independently at rate $K/n$.  

\subsubsection{2NRQ graph update: } \label{s:graphupdate}
We shall now describe 
a recursive stochastic construction of
$E_t$ from $E_{t-1}$, given $r_t$, such that
the sampling property holds.
The idea was presented in Figure \ref{f:2nf}.
A vertex $v''$ of degree two or more in 
$(V, E_{t-1})$ has at least one pair of neighbors. Suppose
$(v, v')$ is such a pair. If $\rho(v, v') > r_t$, do nothing.
If $\rho(v, v') \leq r_t$, then edge $vv'$ is added to $E_t$
according to the success of a Bernoulli$(f_t(v, v'))$ trial,
 where the function $f_t$ is chosen such that the neighbors of $v$
 under $E_t$ are distributed uniformly on $Q_t(v)$, rather than
 biased towards points close to $v$. The formula is as follows.
 
 \begin{defn}
 When edge $vv'$ is proposed for $E_t$, the
 \textbf{acceptance sampling rate} $f_t(v, v')$ is zero if $\rho(v, v') > r_t$, 
 and otherwise is defined as follows.  The volume of the
intersection of the $r$-balls around $v$ and $v'$ is denoted
 \[
 \nu_r(v, v'):=\lambda(B_r(v) \cap B_r(v')); \quad \rho(v, v') \leq r.
 \]
 The minimum value of this volume, over all pairs
 $v, v'$ whose separation does not exceed $s>0$, is
  \begin{equation}\label{e:volumebound}
 g(s, r):= \min_{u, u' \in \mathcal{X}: \rho(u, u') \leq s}\nu_{r}(u, u').
 \end{equation}
 In particular, $g(r_t, r_{t-1})$ is the volume of intersection
 of two $r_{t-1}$-balls around points a distance $r_{t}$ apart, which is
 less than $\nu_{r_{t-1}}(v, v')$ when $\rho(v, v') < r_t$.
 Take
  \begin{equation}\label{e:acceptancerate}
     f_t(v, v'):= 
     \frac{g(r_t, r_{t-1})}
     {\nu_{r_{t-1}}(v, v')} \leq 1.
 \end{equation}
 \end{defn}
 
The function (\ref{e:volumebound}) is strictly positive when $0 < s < 2 r$, by the triangle inequality; We shall compute it explicitly for the torus in (\ref{e:torusbound}).

The purpose of this acceptance sampling computation 
will be apparent in the proof of Proposition \ref{p:validupdate}:
whenever $\rho(v, v') \leq r_t$, the pool of common
neighbors $v''$ of both $v$ and $v'$ in $E_{t-1}$ has a 
Poisson number of elements whose mean is
\begin{equation}\label{e:meancommonnbrs}
\frac{n \theta_{t-1}^2 \nu_{r_{t-1}}(v, v')}{\lambda(\mathcal{X})}
 \end{equation}
 which depends on the distance between $v$ and $v'$;
 however this mean is reduced by acceptance sampling to
\begin{equation}\label{e:commons}
f_t(v, v') \cdot
\frac{n \theta_{t-1}^2 \nu_{r_{t-1}}(v, v')}{\lambda(\mathcal{X})}
= \frac{n \theta_{t-1}^2 g(r_t, r_{t-1})}{\lambda(\mathcal{X})}  
\end{equation}
which is not affected by the actual distance between $v$ and $v'$, if 
$\rho(v, v') \leq r_t$.

It is this non-dependence on distance which gives $E_t$ the desired sampling property.
 After this update we shall compute $\theta_t$ as the mean rate at which
 elements of $Q_t(v)$ occur as neighbors of $v$ under $E_t$.
 
\subsubsection{Value extracted from final graph:} 
 Suppose this procedure can be shown to work in such a way that,
 substituting from (\ref{e:samplerate}),
 \[
 \frac{\theta_t n \lambda(B_{r_t}(v))  }{ \lambda(\mathcal{X}) }
 = K, \quad t = 0, 1, \ldots, T.
 \]
Then for any $v$, the neighbors of $v$ in $E_T$, 
whose expected number is $K$,
form a proportion $\theta_T$ of all the points of $V$ within
distance $r_T$. The key insight is:

\begin{center}
\textit{If $\theta_T \approx 1$, 
the neighbors of $v$ in $E_T$ constitute approximate nearest neighbors of $v$ in $(V, \rho)$.}
\end{center}

Given $K$ and the measure $\lambda$, a successful algorithm means 
designing parameters (\ref{e:radialparams}) so that $\theta_T \approx 1$ and
$T$ is $O(\log{n})$, giving $O(n \log{n})$ total work.
This we will do for the torus example in Theorem \ref{t:torus}.
If $O(n)$ steps were required, for example, the algorithm
would appear to be useless, since $O(n^2)$ total work would be required.

%START HERE!

\subsection{Explicit calculations on the $d$-dimensional torus}
\subsubsection{Metric embedding of $V$}
\label{s:torus}
Let $T^d$ denote the $d$ dimensional hypercube $[-1, 1]^d$ in which
opposite faces are glued together to give the $d$-dimensional torus.
The metric $\rho$ on $T^d$ is induced by the $\ell_{\infty}^d$ norm:
\[
\rho((x_1, \ldots, x_d), (y_1, \ldots, y_d)):=
\max_i \{ \min \{|x_i - y_i|, |x_i - y_i \pm 2| \} \}.
\]
Henceforward arithmetic on $\mathbf{x}:=(x_1, \ldots, x_d)$ will be interpreted 
so that, for each $i$, replacing $x_i$ by $x_i \pm 2$
makes no difference.
The compact metric space $(T^d, \rho)$ has diameter $1$. 
A closed ball
centered at $\mathbf{x}$ with radius $r$ is denoted
$B_r(\mathbf{x}) \subset T^d$. When $r < 1$,
\[
B_r(\mathbf{0}):=[-r, r]^d.
\]

Let $\lambda$ denote Lebesgue measure on the Borel sets of $T^d$,
which has total measure $2^d$.
The volume ratio function (\ref{e:volumeratio}) becomes
\[
h(r):=\frac{(2 r)^d}{2^d} = r^d, \quad r \leq 1.
\]
Our vertex set $V$ will be a random subset of $T^d$ given by a realization of a
homogeneous Poisson point process associated with the measure 
$n 2^{-d} \lambda$.
Thus the cardinality $|V|$ is a Poisson($n$) random variable.

\subsubsection{Intersecting balls on the torus}

\begin{figure}
\caption{
\textit{Suppose $0 < s < r$, and $v, v'$ are two points in
$\R^d$ with the $\ell_{\infty}^d$ norm, separated
by distance $s$. The volume of the set of points $v''$
within distance $r$ of both $v$ and $v'$ is at least $(2 r - s)^d$.
The case $d=2$ is illustrated. Here $0 < t \leq s < r$, and $v$ and $v'$ differ by $s$ in one coordinate, and
$t$ in the other. The locus of points $v''$ has area $(2 r - s) (2 r - t)$,
which is at least $(2 r - s)^2$.
} 
} \label{f:2nf-box}
\begin{center}
\scalebox{0.5}{\includegraphics{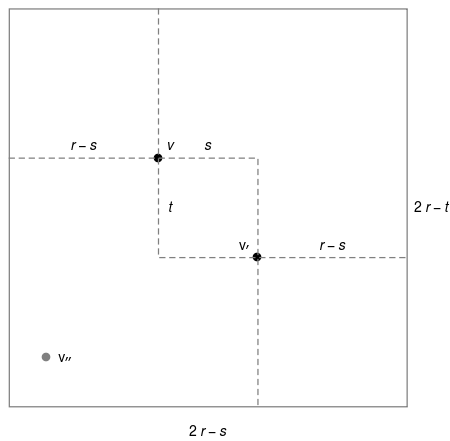}} 
\end{center}
\end{figure}

We shall not prove the following geometric fact, but instead
ask the reader to inspect Figure \ref{f:2nf-box}.

\begin{lem}\label{e:torusballs}
Consider the metric space $T^d$ with the $\ell_{\infty}^d$ norm.
Suppose $0 < s < r < 1$.
 When $\rho(v, v') \leq r$, denote by $\nu_r(v, v')$ the volume of
 the intersection $B_r(v) \cap B_r(v')$, 
 as in Section \ref{s:graphupdate}. Then
  \begin{equation}\label{e:torusbound}
  g(s, r):=
\min_{u, u' : \rho(u, u') \leq s}\nu_{r}(u, u') =
\min \{ 1, (2 r - s)^d \}.
 \end{equation}
\end{lem}

\textbf{Remark: }For another metric, the formula would change.
For example, in the Euclidean metric on $\R^3$, integral calculus
gives
\[
\frac{\pi}{12}(4 r + s) (2 r - s)^2.
\]

\subsection{General update formula for 2NRQ parameters}
We shall now establish the validity of the inductive step
on which 2NRQ depends, and declare the update formulas
for the parameters (\ref{e:radialparams}).
\begin {prop}[2NRQ parameter update]
\label{p:validupdate}
Suppose $V$ is a homogeneous Poisson process, with mean
cardinality $n$, on a compact metric space
$(\mathcal{X},\rho)$ as in
Section \ref{s:compactspace}.
Suppose that, for some $t \geq 1$, and for each $v \in V$.
the neighbors of $v$ under $E_{t-1}$ are a uniform random sample
at rate $\theta_{t-1}$ 
of the elements of $V \setminus \{v\}$ in the ball of radius $r_{t-1}$ centered at $v$,
where $\theta_{t-1}$ and $r_{t-1}$ are parameters connected by
(\ref{e:samplerate}). If there exists $r_t < r_{t-1}$
satisfying the formula:
\begin{equation}\label{e:distanceupdate}
\left(\frac{K}{n h(r_{t-1})}
\right)^2  \frac{n g(r_t, r_{t-1})}{\lambda(\mathcal{X})} = - \log\left(
1 - \frac{K}{n h(r_{t})}
\right),
\end{equation}
use this new radius $r_t$ to
construct $E_t$ according to the 2NRQ update, with acceptance sampling
as described in (\ref{e:acceptancerate}). In that case,
for each $v \in V$,
the neighbors of $v$ under $E_{t}$ are a random sample
at rate $\theta_{t}:=K/(n h(r_{t}))$ 
of the elements of $V \setminus \{v\}$ in the ball of radius $r_{t}$ centered at $v$. 
\end{prop}

\textbf{Remark: } We do not assert that for every $K, n$,
and for every metric space, the equation (\ref{e:distanceupdate}) is solvable for $r_t$.
See Theorem \ref{t:torus}.

\begin{proof}
Assume that the random graph $(V, E_{t-1})$ has the properties described. In our probability calculations, we do not 
condition on any knowledge about the positions of points in
$V$ which might have been gathered in previous 2NRQ steps.
Let $r_t$ be as in (\ref{e:distanceupdate}).
Suppose $v, v'$ are a pair of vertices with separation
$\rho(v, v') \leq r_t$. By the uniformity hypothesis,
the number of vertices
\[
v'' \in B_{r_{t-1}}(v) \cap B_{r_{t-1}}(v')
\]
is a Poisson random variable with mean
\[
\frac{n \nu_{r_{t-1}}(v, v')}{\lambda(\mathcal{X})}.
\]
Any such vertex $v''$ has a probability of 
\[
\theta_{t-1}^2 = \left(\frac{K}{n h(r_{t-1})}\right)^2
\]
of being adjacent to both $v$ and $v'$ under $E_{t-1}$.
Thus the number of common neighbors of $v$ and $v'$
under $E_{t-1}$ is Poisson with the mean given by (\ref{e:meancommonnbrs}), which is reduced by acceptance
sampling to (\ref{e:commons}), which we shall abbreviate to $\mu$. Note that the formula
(\ref{e:commons}) does not depend on the value of
$\rho(v, v')$, provided it is less than $r_t$.
This is the reason why the neighbors of $v$ under $E_t$
are uniformly distributed in the ball of
radius $r_t$ about $v$. 

The probability that no edge between $v$ and $v'$ exists
in $E_t$ is $e^{-\mu}$, which is the mass at zero of a 
Poisson$(\mu)$ random variable. Hence the new rate is
\[
\theta_t:=\frac{K}{n h(r_{t})} = 1 - e^{-\mu}; \quad
\mu = \frac{n \theta_{t-1}^2 g(r_t, r_{t-1})}{\lambda(\mathcal{X})}.
\]
Rewrite as $\mu = - \log{(1 - \theta_t)}$. After the substitution
(\ref{e:samplerate}), this
is exactly the formula (\ref{e:distanceupdate}).

Thus $E_t$ has the desired properties, and the relationship
(\ref{e:samplerate}) between $\theta_t$ and $r_t$ has been
maintained.
\end{proof}

\subsection{2NRQ for uniform points on the torus completes in $O(\log{n})$ steps}
Here is our positive result for 2NRQ on the $d$-dimensional torus.
We believe that, with more effort,  similar results could be obtained for the Euclidean metric
on a bounded subset of $\R^d$, and for geodesic distance on a sphere.

\subsubsection{Formulae for radii:  }\label{s:radii-formulas}
Parameter computation for 2NRQ amounts to solving (\ref{e:distanceupdate}) for $r_t$ in terms of $r_{t-1}$,
for as many steps $t$ as desired, to obtain the full parameter
sequence (\ref{e:radialparams}), where $h(r)=r^d$ on the torus.
Lemma \ref{e:torusballs} shows that, for the torus, 
whenever $2 r_{t-1} - r_t < 1$,
(\ref{e:distanceupdate}) takes the simpler form
\begin{equation}\label{e:torusradii}
  \left(\frac{K}{n r_{t-1}^d}
\right)^2 \frac{n (2 r_{t-1} - r_t)^d}{2^d} = - \log\left(
1 - \frac{K}{n r_{t}^d}
\right).  
\end{equation}
When $2 r_{t-1} - r_t \geq 1$, as occurs when $t=1$, (\ref{e:distanceupdate}) 
specifies $r_t$ explicitly in terms of $r_{t-1}$ via the formula
\begin{equation}\label{e:torusradii-alt}
 - \log\left(1 - \frac{K}{n r_{t}^d}\right) = 
  \left(\frac{K}{n r_{t-1}^d}\right)^2 \frac{n}{2^d}.   
\end{equation}
The proof of Theorem \ref{t:torus} shows that there exists $t' \geq 1$ such that 
(\ref{e:torusradii-alt}) applies when $t \leq t'$, while
(\ref{e:torusradii}) applies when $t > t'$.

\subsubsection{An explicit computation of radii for 2NRQ: } \label{s:explicitradii}
For example when $n = 10^7$, $K=28$, and $d=4$, the criterion
$2 r_{t-1} - r_t \geq 1$ holds only for $t \leq 2$,
giving $r_1 = 0.8694$, $r_2 = 0.6572$. 
For $t \geq 3$, the formula (\ref{e:torusradii}) applies.
Radii $(r_t)$ are shown in the following table, together
with the corresponding values of $(\theta_t)$ (when the latter exceed $10^{-4}$):
\begin{center}
\begin{tabular}{c|c|c|c|c|c|c|c|c|c|}
$t$ & 0 & 1 & 2 & 3 & 4 & 5 & 6 & 7 & 8 \\ \hline
$r_t$ & 1.0 & 0.869 & 0.657 & 0.420 & 0.268 &  0.171 &  0.120 & 0.072& 0.052 \\ \hline
$\theta_t$ & - & - & - & - &  .0005 &  .0032&  .0191 &  .1040 & .3856 \\ \hline
\end{tabular}
\end{center}
In summary, for $n = 10^7$, $K=28$, and $d=4$, eight rounds of
2NRQ lead to a graph $(V, E_8)$ where the neighbors of
an arbitrary vertex $v$ constitute about $39\%$ of the elements
of $V$ within distance $0.052$ of $v$. No ninth round is shown, because
the average vertex degree in $(V, E_9)$ would drop below $K$.

\subsubsection{Parameter estimates: }
To prove Theorem \ref{t:torus}, we shall bound the
radii $(r_t)_{t \geq t'}$ between two geometric progressions,
which allows us to show that the number of steps of 2NRQ is logarithmic in $n$.
For this we prepare some parameter estimates.

The mean number $K$ of neighbors must satisfy a lower bound $K > 2^d$
in terms of the dimension $d$. This allows us to set the
main \emph{scaling parameter}
\begin{equation}\label{e:scaleparam}
  \gamma:=1 - \sqrt{1 - \frac{2} {K^{1/d}}} \in (0,1).  
\end{equation}
Indeed since $K > 2^d$, choose $\beta > 1$ such that $K/\beta > 2^d$, and
define $\gamma_* \in (\gamma, 1)$ via
\begin{equation}\label{e:scaleparam*}
\gamma_*:=
1 - \sqrt{1 - 2 (\beta/K)^{1/d}} \in (0,1).  
\end{equation}
The roles of $\gamma$ and $\gamma_*$ will be that
\(
\gamma < r_t/r_{t-1} \leq \gamma_*
\)
when $t \geq t'$. The proof will mandate that the success rate $\theta_t:=K/(n r_t^d)$ does not
exceed $\alpha \in (0,1)$, implicitly defined by the formula:
\begin{equation}\label{e:alpha}
    \beta:= - \alpha^{-1} \log{(1 - \alpha)}.
\end{equation}
The number $\tau$ of rounds of 2NRQ is a function of the decreasing sequence $(r_t)$:
\begin{equation}\label{e:runtime}
\tau = \max \{t: r_t^d \geq \frac{K}{n \alpha}\}.
\end{equation}
Suppose $\tau > t'$, as will happen for large $n$. Since $\gamma < r_t/r_{t-1} \leq \gamma_*$,
the success rates in rounds $t > t'$ satisfy:
\begin{equation}\label{e:successbounds}
   \gamma_*^{-d} \leq \frac{\theta_t}{\theta_{t-1}} < \gamma^{-d}. 
\end{equation}
It follows that $\alpha \gamma^{d}$ is a lower bound on the success rate,
because if $\theta_{t-1} < \alpha \gamma^{d}$, then another round of 2NRQ would still be permissible.
Here, for example are parameters corresponding to the numerical example of Section 
\ref{s:explicitradii}:
\begin{center}
\begin{tabular}{|c|c|c|c|c|c|c|c|c|}
$n$ & $K$ & $d$ & $\beta$ & $\gamma$ & $\gamma_*$ &  $\alpha$ & $\tau$ & $\alpha \gamma^{d}$  \\ \hline
$10^7$ & 28 & 4 & 1.386 & 0.6387 & 0.7621 & 0.5 & 8 & 0.083 \\
\end{tabular}
\end{center}
The actual final value of the success rate was $0.3856$, which lies in the interval
$(\alpha \gamma^{d}, \alpha]$. Doubling the value of $K$ would allow $\beta = 2.56$,
and a higher success rate of $\alpha= 0.9$.

\subsubsection{Technical result on 2NRQ}
\begin{thm}[Homogeneous Poisson process on torus in $d$ dimensions]
\label{t:torus}
Consider a homogeneous Poisson process associated with the measure
$n 2^{-d} \lambda$ on the 
$d$-dimensional torus, as described in Section \ref{s:torus}.
Assume $K >2^d$, choose $\beta \in (1, K/2^d)$, and
set parameters $0 < \gamma < \gamma_* < 1$ and $\alpha$ according to 
(\ref{e:scaleparam}), (\ref{e:scaleparam*}) and (\ref{e:alpha}). 

\begin{enumerate}
    \item[(a)]
The update formulas (\ref{e:torusradii}) and (\ref{e:torusradii-alt})
have a unique solution for $r_t$, in terms of $r_{t-1}$, starting at $r_0 = 1$, so long as
\(
r_t^d \geq K /(n \alpha),
\)
which allows the 2NRQ algorithm to run for $\tau$ steps as in (\ref{e:runtime}).    
        \item[(b)]
There is some $t'$, depending on $K, \beta$, and $d$, such that
the radii $(r_t)_{t > t'}$ are bounded above and below by geometric sequences:
\(
\gamma < r_t/r_{t-1} \leq \gamma_*.
\)
            \item[(c)]
Suppose $n > K/(\alpha 2^d)$, which implies $\tau >t'$. For each $v \in V$,
the neighbors of $v$ under $E_{\tau}$ are a random sample
at rate $\alpha \gamma^d$ or more
of the elements of $V \setminus \{v\}$ in the ball of radius 
$r_{\tau}$ centered at $v$. The mean number of such neighbors is $K$.
\end{enumerate}

\end{thm}

\textbf{Remark: } Corollary \ref{c:2nfwork} will use the
sub-geometric decay of the radii $(r_t)$ to prove that, as $n$ increases,
$O(\log{n})$ steps of 2NRQ suffice according to (\ref{e:runtime}).

\begin{proof}
On the torus with $\ell_{\infty}^d$ norm, we saw that $h(r) = r^d$.
Lemma \ref{e:torusballs} shows that
\[
g(r_t, r_{t-1}) = \min \{1, (2 r_{t-1} - r_t)^d \}
\]
which implies that distance
update formula (\ref{e:distanceupdate}) takes the form 
(\ref{e:torusradii}) or (\ref{e:torusradii-alt}) on the torus.
To be precise, if $r_{t-1} > 1/2$, first try to evaluate $r_t$ using (\ref{e:torusradii-alt}).
If the result satisfies $2 r_{t-1} - r_t \geq 1$, use this unique value of $r_t$.
If not, we may assume $r_{t-1}< (1 + r_t)/2$, and so we shall use (\ref{e:torusradii}) instead. Since $(r_t)$ is decreasing,
there is some $t' \geq 1$, bounded in (\ref{e:boundt'}) below, such that
the formula (\ref{e:torusradii-alt}) applies for $1 \leq t \leq t'$,
and formula (\ref{e:torusradii}) applies for $t > t'$.

It might happen that $\theta_{t'}:=K/(n r_{t'}^d) > \alpha$,
in which case we can stop. Otherwise we must check the
existence and
uniqueness of the solution of (\ref{e:torusradii}) for $r_t$, in terms of $r_{t-1}< (1 + r_t)/2$, when $t > t'$. This will permit 
application of Proposition \ref{p:validupdate},
which will ensure validity of the 2NRQ algorithm.
In either case, we shall verify geometric decay of the radii.

Recall that $\beta:= - \alpha^{-1} \log{(1 - \alpha)} > 1$. 
Convexity of $x\mapsto -\log{(1-x)}$ guarantees that
\[
x < -\log{(1-x)} \leq \beta x, \quad x \in (0, \alpha].
\]
It follows that, so long as $t > t'$ and $\theta_t:=K/(n r_t^d) \leq \alpha$,
\begin{equation}\label{e:radialbounds}
    \frac{K}{n r_t^d} < 
- \log\left(
1 - \frac{K}{n r_t^d}
\right) =
\frac{K^2}{n r_{t-1}^d} \left(
1- \frac{r_t}{2 r_{t-1}}
\right)^d
\leq \frac{\beta K}{n r_t^d}.
\end{equation}
Cancel terms in the inequalities, and abbreviate $r_t/r_{t-1} < 1$
to $y$, to obtain:
\[
1 < K y^d (1 - y/2)^d \leq \beta.
\]
Divide by $K$, take the $1/d$ power, subtract $1/2$, and
complete the square:
\[
\frac{1}{K^{1/d}} - \frac{1}{2} < - \frac{1}{2} (1 - y)^2
\leq \left(\frac{\beta}{K}\right)^{1/d} - \frac{1}{2}.
\]
Double, change signs, and take the positive square root:
\[
\sqrt{1 - 2 (\beta/K)^{1/d}} \leq 1 - y < 
\sqrt{1 -  2 K^{-1/d} }.
\]
In other words, there is a unique solution $y:=r_t/r_{t-1}<1$
to (\ref{e:distanceupdate}), and this solution satisfies:
\[
\gamma < \frac{r_t}{r_{t-1}} \leq \gamma_*, \quad t \geq t'.
\]
As for the case $1 \leq t \leq t'$, the pair of inequalities
(\ref{e:radialbounds}) takes the simpler form:
\[
    \frac{K}{n r_t^d} < 
- \log\left(
1 - \frac{K}{n r_t^d}
\right) =
\frac{K^2}{n 2^d r_{t-1}^{2 d}} 
\leq \frac{\beta K}{n r_t^d}.
\]
This shows that the ratios $(r_t/r_{t-1})_{1 \leq t \leq t'}$ satisfy an
iteration with the bounds:
\[
r_{t-1} \phi < \frac{r_t}{r_{t-1}}  \leq
r_{t-1} \phi_*,
\quad t = 1, 2, \ldots, t'; \quad 
 \frac{2}{K^{1/d}}=:\phi < \phi_*:= 2 (\beta/K)^{1/d} < 1.
\]
An induction, started at $r_0=1$, shows that
\[
\phi^{2^t - 1} < r_t \leq \phi_*^{2^t - 1}, \quad t = 1, 2, \ldots, t'.
\]
Since $t \leq t'$ can only occur when $r_t > 1/2$, the number of steps where 
(\ref{e:torusradii-alt}) applies is bounded by a constant determined by $K, \beta$, and $d$:
\begin{equation}\label{e:boundt'}
    t' < \log_2{ \left(1 + 1/\log_2(1/\phi_*) \right) }
= \log_2{ \left( \frac{\log_2(K/\beta)}{\log_2(K/\beta) - d}
\right) }.
\end{equation}
The validity of the bound $\theta_{t} \leq \alpha$ for $t = 0 , 1, 2, \ldots, \tau$,
is a consequence of the definition (\ref{e:runtime}).
Thus Proposition \ref{p:validupdate} applies, and the 2NRQ algorithm runs correctly
for $\tau$ steps, at the end of which $\theta_{\tau} \leq \alpha$.
If $\theta_{\tau} \leq \alpha \gamma^d$, then another round of 2NRQ would be possible, as
explained in (\ref{e:successbounds}). Thus $\theta_{\tau} > \alpha \gamma^d$ provided
$\tau > t'$, which occurs if $r_{\tau} < 1/2$, which in turn follows from 
$n > K/(\alpha 2^d)$ by (\ref{e:runtime}).
\end{proof}

\subsubsection{Work estimates}

\begin{cor}[2NRQ work estimate for Poisson process on torus in $d$ dimensions]
\label{c:2nfwork}
The mean number of distance evaluations to achieve a success rate at least $\alpha \gamma^d$ in the second neighbor range query,
with a neighbor average of $K > 2^d$, is bounded above
by a quantity proportional to 
\[
n K^2 \left( t' + \frac{ \log{(n \alpha / K)}} {d \log{(1/\gamma_*)}} \right)
\]
where $t'$ does not depend on $n$,
provided $- \alpha^{-1} \log{(1 - \alpha)}=\beta \in (1, K/2^d)$.
In particular, $O(\log{n})$ rounds of 2NRQ suffice.
\end{cor}

\textbf{Remark: }
It is not easy to describe how the complexity of 2NRQ varies with dimension,
because the quality of the result of the algorithm increases with $\beta$,
i.e. with $K/2^d$.

\begin{proof}
By Theorem \ref{t:torus}, part (b),
\(
r_t < \gamma_*^{t-t'} r_{t'}
\)
for $t = t'+1, 1, \ldots, \tau$. It follows from the definition (\ref{e:runtime}) of $\tau$
that
\(
\gamma_*^{d(\tau-t')}  > \gamma_*^{d(\tau-t')} r_{t'}^d > r_{\tau}^d \geq K/(n \alpha)
\), and thus
\[
\tau  < t' + \frac{ \log{(n \alpha / K)}} {d \log{(1/\gamma_*)}}.
\]
At each step, $O(\binom{K}{2})$ neighbor pairs must be checked,
for each of $n$ vertices. Hence the upper bound on the work.
\end{proof}

\section{Conclusions and Future Work}

We have seen that NND fails to achieve sub-quadratic complexity
for generic concordant ranking systems. It would be interesting to 
know whether a simiar failure occurs for
rank cover trees \cite{hou} and comparison trees \cite{hag}.

The $\lceil\log_{K-1}{n}\rceil$ diameter bound for the expander graph
used to initialize NND suggests that $O(\log{n})$ rounds of 
friend set updates may share enough information for the FOF
principle to work.
Experiments, reported in a sequel \cite{darnnd}, give a class of 
examples where $2 \lceil \log_K{n}\rceil$ rounds of friend set updates 
sufficed for successful 
convergence of NND provided $K$ exceeded the dimension $d$
in which points were embedded.
The second neighbor range query, which exploits FOF
but is not based on rankings, provably finishes in $O(\log{n})$ rounds when applied
to a class of homogeneous Poisson processes on compact separable
metric spaces, but only when $K > 2^d$ (at least for the $\ell_{\infty}$ norm on the torus).

Combinatorial disorder, introduced by Goyal et al.\ \cite{goy},
defines approximate triangle inequalities on ranks.  
These authors define a disorder constant $D$ such that,
in the notation of Definition \ref{d:ranksystem},
\[
r_y(x) \leq D (r_z(x) + r_z(y)), \quad \forall x, y, z.
\]
Possibly some condition on $D$, or a similar notion, would guarantee
 that NND finishes in $O(n \log{n})$ work
with a constant depending on $D$.

Can one say more about the structure of the
 equivalent metrics graph of Section \ref{sec:varphi}? 
 Further insight into its structure may shed light on how frequently NND can be expected to work well in practice. It is possible that our model of a generic CRS is biased in favor of those for which NND has $O(n^2)$ complexity.

\textbf{Acknowledgments: } The authors thank Leland McInnes (TIMC)
for drawing our attention to nearest neighbor descent, and
 Kenneth Berenhaut and Katherine More (WFU) for valuable
 discussions about data science based on rankings.

 \appendix
 \section{Bounding Diameter of the Initial Random Graph}
 \subsection{Main result}
 To motivate Proposition \ref{p:diameter}, observe
 that the arcs of $D$ refer to the initial friend selections
 in NND, while edges of $G$ refer to the union of friend and
 cofriend relations. The diameter of $G$ limits
 the rate of spread of information in NND.
 \begin{prop}
 \label{p:diameter}
Suppose $\Gamma_x$ is a uniform random $K$-subset
of $S \setminus \{x\}$ for each $x \in S$, with
$\{\Gamma_x, x \in S\}$ independent. For $|S| = n$,
consider the regular directed graph $D=D^K_n$ on $S$ with arcs $\bigcup_x \{(x, y), y \in \Gamma_x \}$, and the corresponding undirected graph $G=G^K_n$ on $S$
with edges $\bigcup_x \{\{x, y\}, y \in \Gamma_x \}$.
Let $B=B_n^K$ denote the diameter of $G$. 
For all $K \geq 3$ and $\epsilon>0$,
% there exists $\epsilon > 0$ such that
\[
\lim_{n \to \infty} \Pr \left[ 
B_n^K \leq (1 + \epsilon) \log_{K-1}(n) \right] = 1.
\]
 \end{prop}

\textbf{Remark: } 
Bollob{\'a}s \& de la Vega \cite{bol}
prove that, for any $\epsilon > 0$, the diameter of a uniform $K$-regular random graph, with $K \geq 3$, almost surely
does not exceed the smallest integer $r$ such that
\[
(K-1)^r \geq (2 + \epsilon) K n \log{n}.
\]
Fernholz \& Ramachandran \cite[Theorem 5.1]{fer}
estimate the diameter of graphs
generated uniformly at random according to
a given vertex degree sequence. 
When vertex degrees $X_i$ are sampled independently
with $X_i - K\sim$ Binomial$(n-1, \frac{K}{n-1})$,
\cite[Theorem 5.1]{fer}
implies a diameter estimate of $\log{(n)}/\log{(2 K)} + o(\log{n})$.
This does not apply directly to our
model, in which vertex degrees are negatively dependent,
because they sum to exactly $2 K n$.
Krivelevich \cite[Lemma 8.2]{kri} shows that
expander graphs in general have $O(\log{n})$
diameter.
We give below a proof from first principles.

\subsection{Supporting estimates}
Proposition \ref{p:diameter} is an immediate consequence of two estimates concerning $D$. 
%(At least when $K\geq 3$---for $K=2$ it is still true, but one must work a little harder, and we won't.) 
For $X \subseteq S$, let $N(X) := \{y \in S \setminus X \; : \; \exists x\in X \text{ with } y \in \Gamma_x\}$.

\begin{prop}[Vertex Expanders]
\label{prop:expander1}
With $D$ as in Proposition \ref{p:diameter},
for all $\varepsilon >0$, there exists $\alpha > 0$ (depending also on $K$) such that the following holds with probability tending to 1 as $n \rightarrow \infty$. For all nonempty  $X \subseteq S \text{ with } |X| < \alpha n/\log n$,
\begin{align}
    \label{eq:expander}
    |N(X)| > (K-1-\varepsilon)|X|.
\end{align}

\end{prop}

\begin{proof}[Proof of Proposition \ref{prop:expander1}]
We follow Vadhan \cite[Theorem 4.4]{vad} with minor changes. For $l < \alpha n /\log n$ (with $\alpha$ to be defined below), let $p_l$ be the probability that there is some $X \subseteq S$ of size exactly $l$ that violates \eqref{eq:expander}. We will show that $\sum_{l=1}^{\alpha n/\log n} p_l = o(1)$.

Fix $l$ and $X = \{x_1,x_2,\ldots,x_l\} \subseteq S$, and imagine choosing one by one the elements 
%$y_1^1, y_1^2, \ldots, y_1^K$ 
of $\Gamma_{x_1}$, followed by the elements 
%$y_2^1, y_2^2, \ldots, y_2^K$ 
of $\Gamma_{x_2}$, $\ldots$ followed by the elements
%$y_l^1, y_l^2, \ldots, y_l^K$ 
of $\Gamma_{x_l}$. Call a choice \emph{bad} if the vertex chosen is an element of $X$ or is the same as some previously-chosen vertex. The probability that any particular choice is bad is at most $l(K+1)/n$, even conditioned on all prior choices. Failure of \eqref{eq:expander} for $X$ requires at least $(1+\varepsilon)l$ bad choices, and thus has probability at most \[
\binom{Kl}{(1+\varepsilon)l} \left(\frac{l(K+1)}{n}\right)^{(1+\varepsilon)l}
\]
(the binomial coefficient choosing which choices are bad). Therefore, summing over possible sets $X$ (of size $l$), Stirling's
approximation gives
\begin{align*}
p_l \leq \binom{n}{l} \binom{Kl}{(1+\varepsilon)l} \left(\frac{l(K+1)}{n}\right)^{(1+\varepsilon)l}
&\leq \left(\frac{ne}{l}\right)^l \left(\frac{Kle}{(1+\varepsilon)l}\right)^{(1+\varepsilon)l} \left(\frac{l(K+1)}{n}\right)^{(1+\varepsilon)l} \\
&< \left(\frac{e^{2+\varepsilon} (K+1)^{2+2\varepsilon}}{(1+\varepsilon)^{1+\varepsilon}} \left(\frac{l}{n}\right)^{\varepsilon}\right)^l.
\end{align*}

Now, fix $\alpha = [e^{2+\varepsilon}(K+1)^{2+2\varepsilon}]^{-1/\varepsilon}$ and (for convenience) $\gamma=(1+\varepsilon)^{-(1+\varepsilon)} < 1$. Then for all $l < \alpha n/\log n$, we have $p_l < \gamma^{l}/(\log n)^{l\varepsilon} \leq \gamma^{l}/(\log n)^{\varepsilon}$, so 
%\textcolor{red}{ Should be $p_l < \gamma^{l}/(\log n)^{l \varepsilon}$, but this seems OK.}

\[ \sum_{l=1}^{\alpha n/\log n} p_l \;<\; 
(\log n)^{-\varepsilon}\sum_{l=1}^\infty \gamma^l
\;=\; o(1). 
\qedhere \] 
\end{proof}
%\frac{\sum_{l=1}^\infty \gamma^{l}}{(\log n)^\varepsilon} \;=\; o(1). \qedhere \]

Let $d_D(x,y)$ be the distance from $x$ to $y$ in $D$, and let $N^r(x):=\{y \in S \;:\; d_D(x,y)=r\}$ and $\overline{N}^r(x):=\cup_{i=0}^r N^i(x)$.

\begin{prop}[Diameter Bound]
\label{prop:expander2}
With $D$ as in Lemma \ref{p:diameter} and any $\varepsilon>0$, %there exists $C>0$ (depending on $K$) such that 
the following holds with probability tending to 1 as $n \rightarrow \infty$. For all $x,y \in S$, letting $r = \lceil\log_{(K-1-\varepsilon)}(2n\log n)^{1/2}\rceil$,
either
\begin{equation}\label{eq:expander2}
 \overline{N}^r(x) \cap \overline{N}^r(y) \neq \varnothing   
\end{equation}
or
\(
\exists z \in N^r(x), w\in N^r(y)\)
such that $w\in \Gamma_z$; in other words,
the diameter of $G$ does not exceed $2 r + 1$.

%%%%%%%%%%%%%%%%%%%%%%%%%%%%%%%%%%%%%%%%%%%%%%%%%%%
\begin{comment}
\begin{align}
\label{eq:expander2}
\overline{N}^r(x) \cap \overline{N}^r(y) = \varnothing
%and $|N^r(x)|,|N^s(y)| \in [\sqrt{C n \log n}, \:\alpha n/\log n]$
\qquad \Longrightarrow \qquad
\exists \: z \in N^r(x),\; w\in N^r(y) \text{ such that } w\in \Gamma_z.
\end{align}
\end{comment}
%%%%%%%%%%%%%%%%%%%%%%%%%%%%%%%%%%%%%%%%%%%%%%%%%%%

\end{prop}

\begin{proof}[Proof of Proposition \ref{prop:expander2}]
Fixing any $Y\subseteq S$ and $z \in S\setminus Y$, 
\begin{align*}
\Pr(\Gamma_z \cap Y = \varnothing) &= \left(\frac{n-|Y|}{n-1}\right)
%\left(\frac{n-|Y|-1}{n-2}\right) 
\cdots
\left(\frac{n-|Y|-K+1}{n-K}\right)
< \left(\frac{n-|Y|+1}{n}\right)^K \\
&\qquad = \left(1- \frac{|Y|-1}{n}\right)^K
< \exp[-K(|Y|-1)/n].
\end{align*}

For %$r,s$ in the range stated and 
any $x,y$, Proposition \ref{prop:expander1} implies $|N^r(x)| > \sqrt{2n\log n}$, and likewise for $y$. Fix $x,y$,
and condition on the event that
$\overline{N}^r(x) \cap \overline{N}^r(y)$ is empty; the conditional
probability that there exists no $z \in N^r(x), w\in N^r(y)$
for which $w\in \Gamma_z$ is at most $\exp[-K|N^r(x)||N^r(y)|/n] < \exp[-2K\log n] = n^{-2K}$. This bound is the same for all $x, y$.
Summing over the $n^2$ choices of $x,y$ yields a probability at most $n^{2-2K} = o(1)$ of the existence of a particular $x,y$ for which the
shortest connecting path exceeds length 
$2 r + 1$.
\end{proof}

\subsection{Proof of Proposition \ref{p:diameter}}
To complete the proof of Proposition \ref{p:diameter}, the distance in $G$ between any $x,y \in S$ satisfying %\eqref{eq:expander2}, or with $z,w$ as
one of the alternatives in Proposition \ref{prop:expander2}
%the alternative,
is at most $2r+1$, with $r$ as in Proposition \ref{prop:expander2}. Now 
\[ 2r + 1 < (1+o(1))\frac{\log n}{\log(K-1)+\log(1-\varepsilon)}, \]
which we want to show to be at most 
\[ (1+\epsilon)\frac{\log n}{\log(K-1)}, \]
for some fixed $\epsilon$. Since Proposition \ref{prop:expander2} holds for any $\varepsilon$, we can easily choose $\varepsilon$ sufficiently small.

\begin{comment}
\begin{proof}
We treat the lower bound first. Write $d_D(x,y)$ for distance in $D$ and $d_G(x,y)$ for distance in $G$. For $x \in S$ and nonnegative integer $r$, write $N^r(x) := \{y \in S \mid d_D(x,y) = r\}$.

Imagine that initially none of the sets $\Gamma_x$ have been chosen. Pick $x_0 \in S$ arbitrarily and run the following simple algorithm until some step FAILS:
\begin{enumerate}
    \item Choose the set $\Gamma_{x_r}$.
    \begin{enumerate}
        \item If $\Gamma_{x_r} \cap \left(\cup_{i=0}^r N^i(x_0)\right) = \varnothing$, then pick $x_{r+1}\in \Gamma_{x_r}$ arbitrarily and return to (1).
        \item Else, FAIL.
    \end{enumerate}
\end{enumerate}
If $x_r$ exists, it is easy to see that $d_G(x_0,x_r)=d_D(x_0,x_r)=r$. WRONG

%we will show that almost surely there is a $y\in S$ whose distance from $x_0$ in $G$ is greater than $(1-\epsilon)(\log n)/(\log K)$ by walking away from $x$ one step at a time. 

\end{proof}
\end{comment}

%%%%%%%%%%%%%%%%%%%%%%%%%%%%%%%%%%%%%%%%%%%%%%%%%%%%%%%%%%%%

\end{document}